\documentclass[smallextended]{svjour3}     
\smartqed  
\usepackage{graphicx}
\usepackage{amsfonts}
\usepackage{amssymb}
\usepackage{amsmath}
\usepackage{amsxtra}
\usepackage{latexsym}
\usepackage{mathptmx}      
\begin{document}

\title{\bf  \Large On the discrepancy principle for some
Newton type methods for solving
nonlinear inverse problems
}

\titlerunning{Short form of title}        

\author{Qinian Jin         \and
    Ulrich Tautenhahn 
}

\authorrunning{Short form of author list} 

\institute{Qinian Jin \at
 Department of Mathematics, The University of Texas at Austin,  Austin, Texas 78712, USA\\
\email{qjin@math.utexas.edu}
           \and
Ulrich Tautenhahn \at
Department of Mathematics, University of Applied Sciences Zittau/G\"{o}rlitz, PO Box 1454,
02754 Zittau, Germany\\
          \email{u.tautenhahn@hs-zigr.de}
}


\newtheorem{Assumption}{Assumption}

\def\A{\mathcal A}
\def\B{\mathcal B}

\maketitle

\begin{abstract}
We consider the computation of stable approximations to the exact
solution $x^\dag$ of nonlinear ill-posed inverse problems $F(x)=y$
with nonlinear operators $F:X\to Y$ between two Hilbert spaces $X$
and $Y$ by the Newton type methods
$$
x_{k+1}^\delta=x_0-g_{\alpha_k}
\left(F'(x_k^\delta)^*F'(x_k^\delta)\right) F'(x_k^\delta)^*
\left(F(x_k^\delta)-y^\delta-F'(x_k^\delta)(x_k^\delta-x_0)\right)
$$
in the case that only available data is a noise $y^\delta$ of $y$
satisfying $\|y^\delta-y\|\le \delta$ with a given small noise
level $\delta>0$. We terminate the iteration by the discrepancy
principle in which the stopping index $k_\delta$ is determined as
the first integer such that
$$
\|F(x_{k_\delta}^\delta)-y^\delta\|\le \tau \delta
<\|F(x_k^\delta)-y^\delta\|, \qquad 0\le k<k_\delta
$$
with a given number $\tau>1$. Under certain conditions on
$\{\alpha_k\}$, $\{g_\alpha\}$ and $F$, we prove that
$x_{k_\delta}^\delta$ converges to $x^\dag$ as $\delta\rightarrow
0$ and establish various order optimal convergence rate results.
It is remarkable that we even can show the order optimality under
merely the Lipschitz condition on the Fr\'{e}chet derivative $F'$
of $F$ if $x_0-x^\dag$ is smooth enough.


\end{abstract}

\def\theequation{\thesection.\arabic{equation}}
\catcode`@=12

\section{\bf Introduction}
\setcounter{equation}{0}

In this paper we will consider the nonlinear inverse problems
which can be formulated as the operator equations
\begin{equation}\label{1}
F(x)=y,
\end{equation}
where $F: D(F)\subset X\to Y$ is a nonlinear operator between the
Hilbert spaces $X$ and $Y$ with domain $D(F)$. We will assume that
problem (\ref{1}) is ill-posed in the sense that its solution does
not depend continuously on the right hand side $y$, which is the
characteristic property for most of the inverse problems. Such
problems arise naturally from the parameter identification in
partial differential equations.

Throughout this paper $\|\cdot\|$ and $(\cdot, \cdot)$ denote
respectively the norms and inner products for both the spaces $X$
and $Y$ since there is no confusion. The nonlinear operator $F$ is
always assumed to be Fr\'{e}chet differentiable, the Fr\'{e}chet
derivative of $F$ at $x\in D(F)$ is denoted as $F'(x)$ and
$F'(x)^*$ is used to denote the adjoint of $F'(x)$. We assume that
$y$ is attainable, i.e. problem (1.1) has a solution $x^\dag \in
D(F)$ such that
$$
F(x^\dag) = y.
$$
Since the right hand side is usually obtained by measurement,
thus, instead of $y$ itself, the available data is an
approximation $y^\delta$ satisfying
\begin{equation}\label{1.3}
\|y^\delta-y\|\le \delta
\end{equation}
with a given small noise level $\delta> 0$. Due to the
ill-posedness, the computation of a stable solution of (\ref{1})
from $y^\delta$ becomes an important issue, and the regularization
techniques have to be taken into account.

Many regularization methods have been considered to solve
(\ref{1}) in the last two decades. Tikhonov regularization is one
of the well-known methods that has been studied extensively (see
\cite{SEK93,JH99,TJ03} and the references therein). Due to the
straightforward implementation, iterative methods are also
attractive for solving nonlinear inverse problems. In this paper
we will consider some Newton type methods in which the iterated
solutions $\{x_k^\delta\}$ are defined successively by
\begin{equation}\label{3}
x_{k+1}^\delta=x_0-g_{\alpha_k}
\left(F'(x_k^\delta)^*F'(x_k^\delta)\right) F'(x_k^\delta)^*
\left(F(x_k^\delta)-y^\delta-F'(x_k^\delta)(x_k^\delta-x_0)\right),
\end{equation}
where $x_0^\delta:=x_0$ is an initial guess of $x^\dag$,
$\{\alpha_k\}$ is a given sequence of numbers such that
\begin{equation}\label{4}
\alpha_k>0, \qquad 1\le \frac{\alpha_k}{\alpha_{k+1}}\le r \qquad
\mbox{and}\qquad \lim_{k\rightarrow\infty} \alpha_k=0
\end{equation}
for some constant $r>1$, and $g_\alpha: [0, \infty)\to (-\infty,
\infty)$ is a family of piecewise continuous functions satisfying
suitable structure conditions. The method (\ref{3}) can be derived
as follows. Suppose $x_k^\delta$ is a current iterate, then we may
approximate $F(x)$ by its linearization around $x_k^\delta$, i.e.
$F(x)\approx F(x_k^\delta)+F'(x_k^\delta)(x-x_k^\delta)$. Thus,
instead of (\ref{1}), we have the approximate equation
\begin{equation}\label{n1}
F'(x_k^\delta)(x-x_k^\delta)=y^\delta-F(x_k^\delta).
\end{equation}
If $F'(x_k^\delta)$ has bounded inverse, the usual Newton method
defines the next iterate by solving (\ref{n1}) for $x$. For
nonlinear ill-posed inverse problems, however, $F'(x_k^\delta)$
in general is not invertible. Therefore, we must use linear
regularization methods to solve (\ref{n1}). There are several ways
to do this step. One way is to rewrite (\ref{n1}) as
\begin{equation}\label{n2}
F'(x_k^\delta)
h=y^\delta-F(x_k^\delta)+F'(x_k^\delta)(x_k^\delta-x_0),
\end{equation}
where $h=x-x_0$. Applying the linear regularization method defined
by $\{g_\alpha\}$ we may produce the regularized solution
$h_k^\delta$ by
$$
h_k^\delta =g_{\alpha_k}
\left(F'(x_k^\delta)^*F'(x_k^\delta)\right) F'(x_k^\delta)^*
\left(y^\delta-F(x_k^\delta)+F'(x_k^\delta)(x_k^\delta-x_0)\right).
$$
The next iterate is then defined to be $x_{k+1}^\delta := x_0
+h_k^\delta$ which is exactly the form (\ref{3}).

In order to use $x_k^\delta$ to approximate $x^\dag$, we must
choose the stopping index of iteration properly. Some Newton type
methods that can be casted into the form (\ref{3}) have been
analyzed in \cite{BNS97,K97,KNS08} under a priori stopping rules,
which, however, depend on the knowledge of the smoothness of $x_0
- x^\dag$ that is difficult to check in practice. Thus a wrong
guess of the smoothness will lead to a bad choice of the stopping
index, and consequently to a bad approximation to $x^\dag$.
Therefore, a posteriori rules, which use only quantities that
arise during calculations, should be considered to choose the
stopping index of iteration. One can consult
\cite{BNS97,H97,DES98,Jin00,BH05,KNS08} for several such rules.

One widely used a posteriori stopping rule in the literature of
regularization theory for ill-posed problems is the discrepancy
principle which, in the context of the Newton method (\ref{3}),
defines the stopping index $k_\delta$ to be the first integer such
that
\begin{equation}\label{2.4}
\|F(x_{k_\delta}^\delta)-y^\delta\|\le \tau \delta
<\|F(x_k^\delta)-y^\delta\|, \quad 0\le k<k_\delta,
\end{equation}
where $\tau>1$ is a given number. The method (\ref{3}) with
$g_\alpha(\lambda)=(\alpha+\lambda)^{-1}$ together with
(\ref{2.4}) has been considered in \cite{BNS97,H97}. Note that
when $g_\alpha(\lambda)=(\alpha+\lambda)^{-1}$, the method
(\ref{3}) is equivalent to the iteratively regularized
Gauss-Newton method \cite{B}
\begin{equation}\label{IRGN}
x_{k+1}^\delta=x_k^\delta -\left(\alpha_k I +F'(x_k^\delta)^*
F'(x_k^\delta)\right)^{-1} \left(F'(x_k^\delta)^*
(F(x_k^\delta)-y^\delta)+\alpha_k(x_k^\delta-x_0)\right).
\end{equation}
When $F$ satisfies the condition like
\begin{align}\label{a1}
F'(x)&=R(x, z) F'(z)+Q(x, z),\nonumber\\
\|I-R(x, z)\|&\le C_R\|x-z\|, \qquad \qquad\qquad x, z\in
B_\rho(x^\dag),\\
\|Q(x, z)\|&\le C_Q\|F'(z)(x-z)\|\nonumber,
\end{align}
where $C_R$ and $C_Q$ are two positive constants, for the method
defined by (\ref{IRGN}) and (\ref{2.4}) with $\tau$ being
sufficiently large, it has been shown in \cite{BNS97,H97} that if
$x_0-x^\dag$ satisfies the H\"{o}lder source condition
\begin{equation}\label{7}
x_0-x^\dag=(F'(x^\dag)^*F'(x^\dag))^\nu \omega
\end{equation}
for some $\omega\in X$ and $0\le \nu\le 1/2$, then
$$
\|x_{k_\delta}^\delta-x^\dag\|\le o(\delta^{2\nu/(1+2\nu)});
$$
while if $x_0-x^\dag$ satisfies the logarithmic source condition
\begin{equation}\label{logsource}
x_0-x^\dag=\left(-\log (F'(x^\dag)^*F'(x^\dag))\right)^{-\mu}
\omega
\end{equation}
for some $\omega\in X$ and $\mu>0$, then
$$
\|x_{k_\delta}^\delta-x^\dag\|\le O((-\ln \delta)^{-\mu}).
$$
Unfortunately, except the above results, there is no more result
available in the literature on the general method defined by
(\ref{3}) and (\ref{2.4}).

During the attempt of proving regularization property of the
general method defined by (\ref{3}) and (\ref{2.4}), Kaltenbacher
realized that the arguments in \cite{BNS97,H97} depend heavily on
the special properties of the function
$g_{\alpha}(\lambda)=(\alpha+\lambda)^{-1}$, and thus the
technique therein is not applicable. Instead of the discrepancy
principle (\ref{2.4}), she proposed in \cite{K98} a new a
posteriori stopping rule to terminate the iteration as long as
\begin{equation}\label{d10}
\max \left\{\|F(x_{m_\delta-1}^\delta)-y^\delta\|,
\|F(x_{m_\delta-1}^\delta)+F'(x_{m_\delta-1}^\delta)
(x_{m_\delta}^\delta-x_{m_\delta-1}^\delta) -y^\delta\|
\right\}\le \tau \delta
\end{equation}
is satisfied for the first time, where $\tau>1$ is a given number.
Under the condition like (\ref{a1}), it has been shown that if
$x_0-x^\dag$ satisfies the H\"older source condition (\ref{7}) for
some $\omega\in X$ and $0\le \nu\le 1/2$, then there hold the
order optimal convergence rates
$$
\|x_{m_\delta}^\delta-x^\dag\|\le C_\nu \|\omega\|^{1/(1+2\nu)}
\delta^{2\nu/(1+2\nu)}
$$
if $\{g_\alpha\}$ satisfies some suitable structure conditions,
$\tau$ is sufficiently large and $\|\omega\|$ is sufficiently
small. Note that any result on (\ref{d10}) does not imply that the
corresponding result holds for (\ref{2.4}). Note also that
$k_\delta\le m_\delta-1$ which means that (\ref{d10}) requires
more iterations to be performed. Moreover, the discrepancy
principle (\ref{2.4}) is simpler than the stopping rule
(\ref{d10}). Considering the fact that it is widely used in
practice, it is important to give further investigations on
(\ref{2.4}).

In this paper, we will resume the study of the method defined by
(\ref{3}) and (\ref{2.4}) with completely different arguments.
With the help of the ideas developed in \cite{Jin00,TJ03,Jin08},
we will show that, under certain conditions on $\{g_\alpha\}$,
$\{\alpha_k\}$ and $F$, the method given by (\ref{3}) and
(\ref{2.4}) indeed defines a regularization method for solving
(\ref{1}) and is order optimal for each $0<\nu\le \bar{\nu}-1/2$,
where $\bar{\nu}\ge 1$ denotes the qualification of the linear
regularization method defined by $\{g_\alpha\}$. In particular,
when $x_0-x^\dag$ satisfies (\ref{7}) for $1/2\le \nu\le
\bar{\nu}-1/2$, we will show that the order optimality of
(\ref{3}) and (\ref{2.4}) even holds under merely the Lipschitz
condition on $F'$. This is the main contribution of the present
paper. We point out that our results are valid for any $\tau>1$.
This less restrictive requirement on $\tau$ is important in
numerical computations since the absolute error could increase
with respect to $\tau$.

This paper is organized as follows. In Section 2 we will state
various conditions on $\{g_\alpha\}$, $\{\alpha_k\}$ and $F$, and
then present several convergence results on the methods defined by
(\ref{3}) and (\ref{2.4}). We then complete the proofs of these
main results in Sections 3, 4, and 5. In Section 6, in order to
indicate the applicability of our main results, we verify those
conditions in Section 2 for several examples of $\{g_\alpha\}$
arising from Tikhonov regularization, the iterated Tikhonov
regularization, the Landweber iteration, the Lardy's method, and
the asymptotic regularization.

\section{\bf Assumptions and main results}
\setcounter{equation}{0}

In this section we will state the main results for the method
defined by (\ref{3}) and the discrepancy principle (\ref{2.4}).
Since the definition of $\{x_k^\delta\}$ involves $F$, $g_\alpha$
and $\{\alpha_k\}$, we need to impose various conditions on them.

We start with the assumptions on $g_\alpha$ which is always
assumed to be continuous on $[0, 1/2]$ for each
$\alpha>0$. We will set
$$
r_\alpha(\lambda):=1-\lambda g_\alpha(\lambda),
$$
which is called the residual function associated with $g_\alpha$.

\begin{Assumption} \label{A2.1}
\begin{footnote}
{Recently we realized that (c) can be derived from (a) and (b).}
\end{footnote}
(a) There are positive constants $c_0$ and $c_1$ such that
$$
0< r_\alpha(\lambda)\le 1, \quad r_\alpha(\lambda)\lambda\le
c_0\alpha \quad \mbox{and}
 \quad 0\le g_\alpha(\lambda)\le c_1\alpha^{-1}
$$
for all $\alpha> 0$ and $\lambda\in [0, 1/2]$;

(b) $r_\alpha(\lambda)\le r_\beta(\lambda)$ for any $0<\alpha\le
\beta$ and $\lambda\in [0, 1/2]$;

(c) There exists a constant $c_2>0$ such that
$$
r_\beta(\lambda)-r_\alpha(\lambda) \le
c_2\sqrt{\frac{\lambda}{\alpha}}r_\beta(\lambda)
$$
for any $0<\alpha\le \beta$ and $\lambda\in [0, 1/2]$.
\end{Assumption}

The conditions (a) and (b) in Assumption \ref{A2.1} are standard
in the analysis of linear regularization methods. Assumption
\ref{A2.1}(a) clearly implies
\begin{equation}\label{A2.1.2}
 0\le r_\alpha(\lambda)\lambda^{1/2} \le c_3 \alpha^{1/2}
 \quad \mbox{and} \quad
0\le g_\alpha(\lambda)\lambda^{1/2}\le c_4 \alpha^{-1/2}
\end{equation}
with $c_3\le c_0^{1/2}$ and $c_4 \le c_1^{1/2}$. We emphasize that
direct estimates on $r_{\alpha}(\lambda)\lambda^{1/2}$ and
$g_\alpha(\lambda)\lambda^{1/2}$ could give smaller $c_3$ and
$c_4$. From Assumption \ref{A2.1}(a) it also follows for each
$0\le \nu\le 1$ that $r_\alpha(\lambda)\lambda^\nu\le c_0^{\nu}
\alpha^{\nu}$ for all $\alpha>0$ and $\lambda\in [0, 1/2]$. Thus
the linear regularization method defined by $\{g_\alpha\}$ has
qualification $\bar{\nu}\ge 1$, where, according to \cite{VV}, the
qualification is defined to be the largest number $\bar{\nu}$ with
the property that for each $0\le \nu\le \bar{\nu}$ there is a
positive constant $d_\nu$ such that
\begin{equation}\label{A2.1.3}
r_\alpha(\lambda)\lambda^\nu\le d_\nu \alpha^\nu \quad \mbox{for
all } \alpha>0 \mbox{ and } \lambda\in [0, 1/2].
\end{equation}
Moreover, Assumption \ref{A2.1}(a) implies for every $\mu>0$ that
$$
r_\alpha(\lambda)(-\ln\lambda)^{-\mu}\le
\min\left\{(-\ln\lambda)^{-\mu}, c_0\alpha \lambda^{-1} (-\ln
\lambda)^{-\mu}\right\}
$$
for all $0<\alpha\le \alpha_0$ and $\lambda\in [0, 1/2]$. It is
clear that $(-\ln\lambda)^{-\mu}\le
\left(-\ln(\alpha/(2\alpha_0))\right)^{-\mu}$ for $0\le \lambda\le
\alpha/(2\alpha_0)$. By using the fact that the function
$\lambda\to c_0\alpha\lambda^{-1}(-\ln\lambda)^{-\mu}$ is
decreasing on the interval $(0, e^{-\mu}]$ and is increasing on
the interval $[e^{-\mu}, 1)$, it is easy to show that there is a
positive constant $a_\mu$ such that $c_0\alpha \lambda^{-1}
(-\ln\lambda)^{-\mu}\le a_\mu
\left(-\ln(\alpha/(2\alpha_0))\right)^{-\mu}$ for
$\alpha/(2\alpha_0)\le \lambda\le 1/2$. Therefore for every
$\mu>0$ there is a positive constant $b_\mu$ such that
\begin{equation}\label{8.25.2}
r_\alpha(\lambda)(-\ln \lambda)^{-\mu}\le b_\mu \left(-\ln
(\alpha/(2\alpha_0))\right)^{-\mu}
\end{equation}
for all $0<\alpha\le \alpha_0$ and $\lambda \in [0, 1/2]$. This
inequality will be used to derive the convergence rate when
$x_0-x^\dag$ satisfies the logarithmic source condition
(\ref{logsource})

The condition (c) in Assumption \ref{A2.1} seems to appear here
for the first time. It is interesting to note that one can verify
it for many well-known linear regularization methods. Moreover,
the conditions (b) and (c) have the following important
consequence.

\begin{lemma}\label{L10.1}
Under the conditions (b) and (c) in Assumption \ref{A2.1}, there
holds
\begin{equation}
\|[r_\beta(A^*A)-r_\alpha(A^*A)]x\| \le
\|\bar{x}-r_\beta(A^*A)x\|+\frac{c_2}{\sqrt{\alpha}}\|A \bar{x}\|
\end{equation}
for all $x, \bar{x}\in X$, any $0<\alpha\le \beta$ and any bounded
linear operator $A:X\to Y$ satisfying $\|A\|\le 1/\sqrt{2}$.
\end{lemma}

\begin{proof}
For any $0<\alpha\le \beta$ we set
\begin{align*}
p_{\beta, \alpha}(\lambda)
:=\frac{r_\beta(\lambda)-r_\alpha(\lambda)} {r_\beta(\lambda)},
\qquad \lambda\in [0, 1/2].
\end{align*}
It follows from the conditions (a) and (b) in Assumption
\ref{A2.1} that
\begin{equation}\label{3.6}
0\le p_{\beta, \alpha}(\lambda) \le \min\left\{1,
c_2\sqrt{\frac{\lambda}{\alpha}}\right\}.
\end{equation}
Therefore, for any $x, \bar{x}\in X$,
\begin{align}\label{10.10}
\|[ r_\beta(A^*A)-r_\alpha(A^*A)]x\|
&=\|p_{\beta, \alpha}(A^*A) r_\beta(A^*A) x\|\nonumber\\
&\le  \|p_{\beta,\alpha}(A^*A)[r_\beta(A^*A)x-\bar{x}]\|
+\|p_{\beta, \alpha}(A^*A) \bar{x}\|\nonumber\\
& \le \|r_\beta(A^*A) x-\bar{x}\|+\|p_{\beta, \alpha}(A^*A)
\bar{x}\|.
\end{align}
Let $\{E_\lambda\}$ be the spectral family generated by $A^*A$.
Then it follows from (\ref{3.6}) that
\begin{align*}
\|p_{\beta, \alpha}(A^*A) \bar{x}\|^2
&=\int_0^{1/2} \left[p_{\beta, \alpha}(\lambda)\right]^2 d\|E_\lambda \bar{x}\|^2\\
&\le c_2^2 \int_0^{1/2} \frac{\lambda}{\alpha} d\|E_\lambda
\bar{x}\|^2
=\frac{c_2^2}{\alpha} \|(A^*A)^{1/2} \bar{x}\|^2\\
&=\frac{c_2^2}{\alpha}\|A \bar{x}\|^2.
\end{align*}
Combining this with (\ref{10.10}) gives the desired assertion. \hfill $\Box$
\end{proof}

For the sequence of positive numbers $\{\alpha_k\}$, we will
always assume that it satisfies (\ref{4}). Moreover, we need also
the following condition on $\{\alpha_k\}$ interplaying with
$r_\alpha$.

\begin{Assumption}\label{A2.2}
There is a constant $c_5>1$ such that
$$
r_{\alpha_k}(\lambda)\le c_5 r_{\alpha_{k+1}}(\lambda)
$$
for all $k$ and $\lambda\in [0, 1/2]$.
\end{Assumption}

We remark that for some $\{g_\alpha\}$ Assumption \ref{A2.2} is an
immediate consequence of (\ref{4}). However, this is not always
the case; in some situations, Assumption \ref{A2.2} indeed imposes
further conditions on $\{\alpha_k\}$. As a rough interpretation,
Assumption \ref{A2.2} requires for any two successive iterated
solutions the errors do not decrease dramatically. This may be
good for the stable numerical implementations of ill-posed
problems although it may require more iterations to be performed.
Note that Assumption \ref{A2.2} implies
\begin{equation}\label{2.31}
 \|r_{\alpha_k}(A^*A) x\|\le c_5\|r_{\alpha_{k+1}}(A^*A) x\|
\end{equation}
for any $x\in X$ and any bounded linear operator $A: X\to Y$
satisfying $\|A\|\le 1/\sqrt{2}$.

Throughout this paper, we will always assume that the nonlinear
operator $F: D(F)\subset X\to Y$ is Fr\'{e}chet differentiable
such that
\begin{equation}\label{2.1}
B_\rho(x^\dag)\subset D(F) \quad \mbox{ for some } \rho> 0
\end{equation}
and
\begin{equation}\label{2.3}
\|F'(x)\|\le \min\left\{c_3\alpha_0^{1/2}, \beta_0^{1/2}\right\},
\qquad x\in B_\rho(x^\dag),
\end{equation}
where $0<\beta_0\le 1/2$ is a number such that
$r_{\alpha_0}(\lambda)\ge 3/4$ for all $\lambda\in [0, \beta_0]$.
Since $r_{\alpha_0}(0)=1$, such $\beta_0$ always exists. The
scaling condition (\ref{2.3}) can always be fulfilled by rescaling
the norm in $Y$.

The convergence analysis on the method defined by (\ref{3}) and
(\ref{2.4}) will be divided into two cases:
\begin{enumerate}
\item[(i)] $x_0-x^\dag$ satisfies (\ref{7}) for some $\nu\ge 1/2$;

\item[(ii)] $x_0-x^\dag$ satisfies (\ref{7}) with $0\le \nu<1/2$
or (\ref{logsource}) with $\mu>0$.
\end{enumerate}
Thus different structure conditions on $F$ will be assumed in
order to carry out the arguments. It is remarkable to see that for
case (i) the following Lipschitz condition on $F'$ is enough for
our purpose.

\begin{Assumption}\label{Lip}
There exists a constant $L$ such that
\begin{equation}\label{2.2}
\|F'(x)-F'(z)\|\le L\|x-z\|
\end{equation}
for all $x, z\in B_\rho(x^\dag)$.
\end{Assumption}

As the immediate consequence of Assumption \ref{Lip}, we have
$$
\|F(x)-F(z)-F'(z)(x-z)\|\le \frac{1}{2} L\|x-z\|^2
$$
for all $x, z\in B_\rho(x^\dag)$. We will use this consequence
frequently in this paper.

During the convergence analysis of (\ref{3}), we will meet some
terms involving operators such as $r_{\alpha_k}(F'(x_k^\delta)^*
F'(x_k^\delta))$. In order to make use of the source conditions
(\ref{7}) for $x_0-x^\dag$, we need to switch these operators with
$r_{\alpha_k}(F'(x^\dag)^* F'(x^\dag))$. Thus we need the
following commutator estimates involving $r_\alpha$ and
$g_\alpha$.

\begin{Assumption} \label{A2.3}
There is a constant $c_6>0$ such that
\begin{equation}\label{3.41}
\|r_\alpha(A^*A)-r_\alpha(B^*B)\|\le c_6 \alpha^{-1/2}\|A-B\|,
\end{equation}
\begin{equation}\label{3.40}
\|\left[r_\alpha(A^*A)-r_\alpha(B^*B)\right]B^*\|\le c_6\|A-B\|,
\end{equation}
\begin{equation}\label{3.42}
\|A\left[r_\alpha(A^*A)-r_\alpha(B^*B)\right]B^*\|\le
c_6\alpha^{1/2}\|A-B\|,
\end{equation}
and
\begin{equation}\label{3.43}
\|\left[g_\alpha(A^*A)-g_\alpha(B^*B)\right]B^*\|\le c_6
\alpha^{-1}\|A-B\|
\end{equation}
for any $\alpha>0$ and any bounded linear operators $A, B: X\to Y$
satisfying $\|A\|, \|B\|\le 1/\sqrt{2}$.
\end{Assumption}

This assumption looks restrictive. However, it is interesting to
note that for several important examples we indeed can verify it
easily, see Section 6 for details. Moreover, in our applications,
we only need Assumption \ref{A2.3} with $A=F'(x)$ and $B=F'(z)$
for $x, z\in B_\rho(x^\dag)$, which is trivially satisfied when
$F$ is linear.

Now we are ready to state the first main result of this paper.

\begin{theorem}\label{T4.1}
Let $\{g_\alpha\}$ and $\{\alpha_k\}$ satisfy Assumption
\ref{A2.1}, (\ref{4}), Assumption \ref{A2.2}, and Assumption
\ref{A2.3}, let $\bar{\nu}\ge 1$ be the qualification of the
linear regularization method defined by $\{g_\alpha\}$, and let
$F$ satisfy (\ref{2.1}), (\ref{2.3}) and Assumption \ref{Lip} with
$\rho>4\|x_0-x^\dag\|$. Let $\{x_k^\delta\}$ be defined by
(\ref{3}) and let $k_\delta$ be the first integer satisfying
(\ref{2.4}) with $\tau>1$. Let $x_0-x^\dag$ satisfy (\ref{7}) for
some $\omega\in X$ and $1/2 \le \nu\le \bar{\nu}-1/2$. Then
$$
\|x_{k_\delta}^\delta-x^\dag\| \le C_\nu
\|\omega\|^{1/(1+2\nu)}\delta^{2\nu/(1+2\nu)}
$$
if $L\|u\|\le \eta_0$, where $u\in {\mathcal
N}(F'(x^\dag)^*)^\perp\subset Y$ is the unique element such that
$x_0-x^\dag=F'(x^\dag)^* u$, $\eta_0>0$ is a constant depending
only on $r$, $\tau$ and $c_i$, and $C_\nu$ is a positive constant
depending only on $r$, $\tau$,  $\nu$ and $c_i$, $i=0, \cdots, 6$.
\end{theorem}

Theorem \ref{T4.1} tells us that, under merely the Lipschitz
condition on $F'$, the method (\ref{3}) together with (\ref{2.4})
indeed defines an order optimal regularization method for each
$1/2\le \nu\le \bar{\nu}-1/2$; in case the regularization method
defined by $\{g_\alpha\}$ has infinite qualification the
discrepancy principle (\ref{2.4}) provides order optimal
convergence rates for the full range $\nu \in [ 1/2, \infty )$.
This is one of the main contribution of the present paper.

We remark that under merely the Lipschitz condition on $F'$ we are
not able to prove the similar result as in Theorem \ref{T4.1} if
$x_0-x^\dag$ satisfies weaker source conditions, say (\ref{7}) for
some $\nu<1/2$. Indeed this is still an open problem in the
convergence analysis of regularization methods for nonlinear
ill-posed problems. In order to pursue the convergence analysis
under weaker source conditions, we need stronger conditions on $F$
than Assumption \ref{Lip}. The condition (\ref{a1}) has been used
in \cite{BNS97,H97} to establish the regularization property of
the method defined by (\ref{IRGN}) and (\ref{2.4}), where the
special properties of $g_\alpha(\lambda)=(\lambda+\alpha)^{-1}$
play the crucial roles. In order to study the general method
(\ref{3}) under weaker source conditions, we need the following
two conditions on $F$.

\begin{Assumption}\label{F1}
There exists a positive constant $K_0$ such that
\begin{align*}
F'(x)&=F'(z) R(x, z),\\
\|I-R(x, z)\|&\le K_0\|x-z\|
\end{align*}
for any $x, z \in B_\rho(x^\dag)$.
\end{Assumption}

\begin{Assumption}\label{F2}
There exist positive constants $K_1$ and $K_2$ such that
\begin{align*}
\|[F'(x)-F'(z)]w\|\le K_1\|x-z\|\|F'(z)w\|+K_2\|F'(z)(x-z)\|\|w\|
\end{align*}
for any $x, z\in B_\rho(x^\dag)$ and $w\in X$.
\end{Assumption}

Assumption \ref{F1} has been used widely in the literature of
nonlinear ill-posed problems (see \cite{SEK93,JH99,Jin00,TJ03});
it can be verified for many important inverse problems. Another
frequently used assumption on $F$ is (\ref{a1}) which is indeed
quite restrictive. It is clear that Assumption \ref{F2} is a
direct consequence of (\ref{a1}). In order to illustrate that
Assumption \ref{F2} could be weaker than (\ref{a1}), we consider
the identification of the parameter $c$ in the boundary value
problem
\begin{equation}\label{8.27.1}
\left\{\begin{array}{ll} -\Delta u+c u=f \qquad& \mbox{in } \Omega\\
u=g \qquad& \mbox{on } \partial \Omega
\end{array}\right.
\end{equation}
from the measurement of the state $u$, where $\Omega\subset
{\mathbb R}^n, n\le 3,$ is a bounded domain with smooth boundary
$\partial \Omega$, $f\in L^2(\Omega)$ and $g\in H^{3/2}(\partial
\Omega)$. We assume $c^\dag\in L^2(\Omega)$ is the sought solution.
This problem reduces to solving an equation of the form
(\ref{1}) if we define the nonlinear operator $F$ to be the
parameter-to-solution mapping $F: L^2(\Omega)\to L^2(\Omega),
F(c):=u(c)$ with $u(c)\in H^2(\Omega)\subset L^2(\Omega)$ being
the unique solution of (\ref{8.27.1}). Such $F$ is well-defined on
$$
D(F):=\left\{c\in L^2(\Omega): \|c-\hat{c}\|_{L^2}\le \gamma
\mbox{ for some } \hat{c}\ge 0 \mbox{ a.e.}\right\}
$$
for some positive constant $\gamma>0$. It is well-known that $F$
has Fr\'{e}chet derivative
\begin{equation}\label{8.27.2}
F'(c) h=-A(c)^{-1}(hF(c)), \qquad h\in L^2(\Omega),
\end{equation}
where $A(c): H^2\cap H_0^1\to L^2$ is defined
by $A(c)u:=-\Delta u+c u$ which is an isomorphism uniformly in a ball
$B_\rho(c^\dag)\subset D(F)$ around $c^\dag$. Let $V$ be the dual space of $H^2\cap H_0^1$ with
respect to the bilinear form $ \langle \varphi, \psi\rangle =\int_\Omega
\varphi(x) \psi(x) dx$.  Then $A(c)$ extends to an isomorphism from
$L^2(\Omega)$ to $V$. Since (\ref{8.27.2}) implies for any $c,
d\in B_\rho(c^\dag)$ and $h\in L^2(\Omega)$
$$
\left(F'(c)-F'(d)\right) h=-A(c)^{-1}\left((c-d)F'(d) h\right)
-A(c)^{-1}\left(h(F(c)-F(d))\right),
$$
and since $L^1(\Omega)$ embeds into $V$ due to the restriction $n\le 3$,
we have
\begin{align}\label{8.27.3}
\|(F'(c)-F'(d))h\|_{L^2} &\le \|A(c)^{-1}\left((c-d)F'(d)
h\right)\|_{L^2}+\|A(c)^{-1}\left(h(F(c)-F(d))\right)\|_{L^2}\nonumber\\
&\le C \|(c-d)F'(d)h\|_{V}+C\|h(F(c)-F(d))\|_{V} \nonumber\\
&\le C \|(c-d)F'(d)h\|_{L^1}+C\|h(F(c)-F(d))\|_{L^1}\nonumber\\
&\le C \|c-d\|_{L^2}\|F'(d)h\|_{L^2}+ C\|F(c)-F(d)\|_{L^2}
\|h\|_{L^2}.
\end{align}
On the other hand, note that
$F(c)-F(d)=-A(d)^{-1}\left((c-d)F(c)\right)$, by using
(\ref{8.27.2}) we obtain
$$
F(c)-F(d)-F'(d)(c-d)=-A(d)^{-1}\left((c-d)\left(F(c)-F(d)\right)\right).
$$
Thus, by a similar argument as above,
$$
\|F(c)-F(d)-F'(d)(c-d)\|_{L^2} \le C
\|c-d\|_{L^2}\|F(c)-F(d)\|_{L^2}.
$$
Therefore, if $\rho>0$ is small enough, we have
$\|F(c)-F(d)\|_{L^2}\le C \|F'(d)(c-d)\|_{L^2}$, which together
with (\ref{8.27.3}) verifies Assumption \ref{F2}. The validity of
(\ref{a1}), however, requires $u(c)\ge \kappa>0$ for all $c\in
B_\rho(c^\dag)$, see \cite{HNS95}.

In our next main result, Assumption \ref{F1} and Assumption
\ref{F2} will be used to derive estimates related to $x_k^\delta
-x^\dag$ and $F'(x^\dag)(x_k^\delta-x^\dag)$ respectively.
Although Assumption \ref{F2} does not explore the full strength of
(\ref{a1}), the plus of Assumption \ref{F1} could make our
conditions stronger than (\ref{a1}) in some situations. One
advantage of the use of Assumption \ref{F1} and Assumption
\ref{F2}, however, is that we can carry out the analysis on the
discrepancy principle (\ref{2.4}) for any $\tau>1$, in contrast to
those results in \cite{BNS97,H97} where $\tau$ is required to be
sufficiently large. It is not yet clear if only one of the above
two assumptions is enough for our purpose. From Assumption
\ref{F2} it is easy to see that
\begin{equation}\label{F2.1}
\|F(x)-F(z)-F'(z)(x-z)\|\le
\frac{1}{2}(K_1+K_2)\|x-z\|\|F'(z)(x-z)\|
\end{equation}
and
\begin{equation}\label{F2.2}
\|F(x)-F(z)-F'(z)(x-z)\|\le
\frac{3}{2}(K_1+K_2)\|x-z\|\|F'(x)(x-z)\|.
\end{equation}
for any $x, z\in B_\rho(x^\dag)$.

We still need to deal with some commutators involving $r_\alpha$.
The structure information on $F$ will be incorporated into such
estimates. Thus, instead of Assumption \ref{A2.3}, we need the
following strengthened version.

\begin{Assumption}\label{A6.2}
(a) Under Assumption \ref{F1}, there exists a positive constant
$c_7$ such that
\begin{align}\label{6.2.1}
\left\|r_\alpha\left(F'(x)^*F'(x)\right)
-r_\alpha\left(F'(z)^*F'(z)\right)\right\| \le c_7 K_0\|x-z\|
\end{align}
for all $x, z\in B_\rho(x^\dag)$ and all $\alpha>0$.

(b) Under Assumption \ref{F1} and Assumption \ref{F2}, there
exists a positive constant $c_8$ such that
\begin{align}\label{6.2.2}
\|F'(x)&\left[r_\alpha\left(F'(x)^*F'(x)\right)
-r_\alpha \left(F'(z)^*F'(z)\right)\right]\| \nonumber\\
&\le c_8(K_0+ K_1) \alpha^{1/2} \|x-z\| +c_8
K_2\left(\|F'(x)(x-z)\|+\|F'(z)(x-z)\|\right)
\end{align}
for all $x, z\in B_\rho(x^\dag)$ and all $\alpha>0$.
\end{Assumption}

Now we are ready to state the second main result in this paper
which in particular says that the method (\ref{3}) together with
the discrepancy principle (\ref{2.4}) defines an order optimal
regularization method for each $0<\nu\le \bar{\nu}-1/2$ under
stronger conditions on $F$. We will fix a constant $\gamma_1>c_3
r^{1/2}/(\tau-1)$.

\begin{theorem}\label{T4.5}
Let $\{g_\alpha\}$ and $\{\alpha_k\}$ satisfy Assumption
\ref{A2.1}, (\ref{4}), Assumption \ref{A2.2}  and Assumption
\ref{A6.2}, let $\bar{\nu}\ge 1$ be the qualification of the
linear regularization method defined by $\{g_\alpha\}$,  and let
$F$ satisfy (\ref{2.1}),  (\ref{2.3}), Assumption \ref{F1} and
Assumption \ref{F2} with $\rho>2(1+c_4\gamma_1)\|x_0-x^\dag\|$.
Let $\{x_k^\delta\}$ be defined by (\ref{3}) and let $k_\delta$ be
the first integer satisfying (\ref{2.4}) with $\tau>1$. Then there
exists a constant $\eta_1>0$ depending only on $r$, $\tau$ and
$c_i$, $i=0, \cdots, 8$, such that if
$(K_0+K_1+K_2)\|x_0-x^\dag\|\le \eta_1$ then

(i) If $x_0-x^\dag$ satisfies the H\"{o}lder source condition
(\ref{7}) for some $\omega\in X$ and $0< \nu\le \bar{\nu}-1/2$,
then
\begin{equation}\label{T4.5.1}
\|x_{k_\delta}^\delta-x^\dag\| \le C_\nu
\|\omega\|^{1/(1+2\nu)}\delta^{2\nu/(1+2\nu)},
\end{equation}
where $C_\nu$ is a constant depending only on $r$, $\tau$, $\nu$
and $c_i$, $i=0, \cdots, 8$.

 (ii) If $x_0-x^\dag$ satisfies the logarithmic source
condition (\ref{logsource}) for some $\omega\in X$ and $\mu>0$,
then
\begin{equation}\label{T4.5.2}
\|x_{k_\delta}^\delta-x^\dag\|\le C_\mu
\|\omega\|\left(1+\left|\ln
\frac{\delta}{\|\omega\|}\right|\right)^{-\mu},
\end{equation}
where $C_\mu$ is a constant depending only on $r$, $\tau$, $\mu$,
and $c_i$, $i=0, \cdots, 8$.
\end{theorem}

In the statements of Theorem \ref{T4.1} and Theorem \ref{T4.5},
the smallness of $L\|u\|$ and $(K_0+K_1+K_2)\|x_0-x^\dag\|$ are
not specified. However, during the proof of Theorem \ref{T4.1}, we
indeed will spell out all the necessary smallness conditions on
$L\|u\|$. For simplicity of presentation, we will not spell out
the smallness conditions on $(K_0+K_1+K_2)\|x_0-x^\dag\|$ any
more; the readers should be able to figure out such conditions
without any difficulty.

Note that, without any source condition on $x_0-x^\dag$, the above
two theorems do not give the convergence of $x_{k_\delta}^\delta$
to $x^\dag$. The following theorem says that
$x_{k_\delta}^\delta\rightarrow x^\dag$ as $\delta\rightarrow 0$
provided $x_0-x^\dag\in {\mathcal N}(F'(x^\dag))^\perp$. In fact,
it tells more, it says that the convergence rates can even be
improved to $o(\delta^{2\nu/(1+2\nu)})$ if $x_0-x^\dag$ satisfies
(\ref{7}) for $0\le \nu<\bar{\nu}-1/2$.

\begin{theorem}\label{T8.1}
(i) Let all the conditions in Theorem \ref{T4.1} be fulfilled. If
$\bar{\nu}>1$ and $x^\dag-x_0$ satisfies the H\"{o}lder source
condition (\ref{7}) for some $\omega\in {\mathcal
N}(F'(x^\dag))^\perp$ and $1/2\le \nu< \bar{\nu}-1/2$, then
$$
\|x_{k_\delta}^\delta-x^\dag\|\le o(\delta^{2\nu/(1+2\nu)})
$$
as $\delta\rightarrow 0$.

(ii) Let all the conditions in Theorem \ref{T4.5} be fulfilled. If
$x_0-x^\dag$ satisfies (\ref{7}) for some $\omega\in {\mathcal
N}(F'(x^\dag))^\perp$ and $0\le \nu< \bar{\nu}-1/2$, then
$$
\|x_{k_\delta}^\delta-x^\dag\|\le o(\delta^{2\nu/(1+2\nu)})
$$
as $\delta\rightarrow 0$.
\end{theorem}

Theorem \ref{T4.1}, Theorem \ref{T4.5} and Theorem \ref{T8.1} will
be proved in Sections 3, 4 and 5 respectively. In the following we
will give some remarks.

\begin{remark}
A comprehensive overview on iterative regularization methods for
nonlinear ill-posed problems may be found in the recent book
\cite{KNS08}. In particular, convergence and convergence rates for
the general method (1.3) are obtained in \cite[Theorem
4.16]{KNS08} in case of a priori stopping rules under
suitable nonlinearity assumptions on $F$.
\end{remark}

\begin{remark}
In \cite{T97} Tautenhahn introduced a general regularization
scheme for (\ref{1}) by defining the regularized solutions
$x_\alpha^\delta$ as a fixed point of the nonlinear equation
\begin{equation}
x=x_0-g_\alpha\left(F'(x)^*F'(x)\right)F'(x)^*
\left(F(x)-y^\delta-F'(x)(x-x_0)\right),
\end{equation}
where $\alpha>0$ is the regularization parameter. When $\alpha$ is
determined by a Morozov's type discrepancy principle, it was shown
in \cite{T97} that the method is order optimal for each $0<\nu\le
\bar{\nu}/2$ under certain conditions on $F$. We point out that
the technique developed in the present paper can be used to
analyze such method; indeed we can even show that, under merely
the Lipschitz condition on $F'$, the method in \cite{T97} is order
optimal for each $1/2\le \nu\le \bar{\nu}-1/2$, which improves the
corresponding result.
\end{remark}

\begin{remark}
Alternative to (\ref{3}), one may consider the inexact Newton type
methods
\begin{equation}\label{8.25.2008}
x_{k+1}^\delta =
x_k^\delta-g_{\alpha_k}\left(F'(x_k^\delta)^*F'(x_k^\delta)\right)
F'(x_k^\delta)^*\left(F(x_k^\delta)-y^\delta\right)
\end{equation}
which can be derived by applying the regularization method defined
by $\{g_\alpha\}$ to (\ref{n1}) with the current iterate
$x_k^\delta$ as an initial guess. Such methods have first been
studied by Hanke in \cite{H97a,H97b} where the regularization
properties of the Levenberg-Marquardt algorithm and the Newton-CG
algorithm have been established without giving convergence rates
when the sequence $\{\alpha_k\}$ is chosen adaptively during
computation and the discrepancy principle is used as a stopping
rule. The general methods (\ref{8.25.2008}) have been considered
later by Rieder in \cite{R99,R01}, where $\{\alpha_k\}$ is
determined by a somewhat different adaptive strategy; certain
sub-optimal convergence rates have been derived when $x_0-x^\dag$
satisfies (\ref{7}) with $\eta<\nu\le 1/2$ for some
problem-dependent number $0<\eta<1/2$, while it is not yet clear if the
convergence can be established under weaker source conditions. The
convergence analysis of (\ref{8.25.2008}) is indeed far from
complete. The technique in the present paper does not work for
such methods.
\end{remark}

Throughout this paper we will use $\{x_k\}$ to denote the iterated
solutions defined by (\ref{3}) corresponding to the noise free
case. i.e.
\begin{equation}\label{1.8}
 x_{k+1}=x_0-g_{\alpha_k}\left(F'(x_k)^*F'(x_k)\right)F'(x_k)^*
\left(F(x_k)-y-F'(x_k)(x_k-x_0)\right).
\end{equation}
We will also use the notations
\begin{align*}
\A:=F'(x^\dag)^*F'(x^\dag), & \quad \A_k:=F'(x_k)^*F'(x_k),  \quad
\A_k^\delta:=F'(x_k^\delta)^* F'(x_k^\delta),\\
\B:=F'(x^\dag)F'(x^\dag)^*, & \quad \B_k:=F'(x_k)F'(x_k)^*,  \quad
\B_k^\delta:=F'(x_k^\delta) F'(x_k^\delta)^*,
\end{align*}
and
\begin{eqnarray*}
e_k:=x_k-x^\dag, \qquad  e_k^\delta:=x_k^\delta-x^\dag.
\end{eqnarray*}
For ease of exposition, we will use $C$ to denote a generic
constant depending only on $r$. $\tau$ and $c_i$, $i=0, \cdots,
8$, we will also use the convention $\Phi\lesssim \Psi$ to mean
that $\Phi\le C \Psi$ for some generic constant $C$. Moreover,
when we say $L\|u\|$ (or $(K_0+K_1+K_2)\|e_0\|$) is sufficiently
small we will mean that $L\|u\|\le \eta$ (or
$(K_0+K_1+K_2)\|e_0\|\le \eta$) for some small positive constant
$\eta$ depending only on $r$, $\tau$ and $c_i$, $i=0, \cdots, 8$.

\section{\bf Proof of Theorem \ref{T4.1}}
\setcounter{equation}{0}

In this section we will give the proof of Theorem \ref{T4.1}. The
main idea behind the proof consists of the following steps:\\


$\bullet$ Show the method defined by (\ref{3}) and (\ref{2.4}) is
well-defined.

$\bullet$ Establish the stability estimate
$\|x_k^\delta-x_k\|\lesssim  \delta/\sqrt{\alpha_k}$. This enables
us to write $\|e_{k_\delta}^\delta\|\lesssim
\|e_{k_\delta}\|+\delta/\sqrt{\alpha_{k_\delta}}$.

$\bullet$ Establish $\alpha_{k_\delta}\ge
C_\nu(\delta/\|\omega\|)^{2/(1+2\nu)}$ under the source condition
(\ref{7}) for $1/2\le \nu\le \bar{\nu}-1/2$. This is an easy step
although it requires nontrivial arguments.

$\bullet$ Show $\|e_{k_\delta}\|\le C_\nu
\|\omega\|^{1/(1+2\nu)}\delta^{2\nu/(1+2\nu)}$, which is the hard
part in the whole proof. In order to achieve this, we pick an
integer $\bar{k}_\delta$ such that $k_\delta\le \bar{k}_\delta$
and $\alpha_{\bar{k}_\delta}\sim
(\delta/\|\omega\|)^{2/(1+2\nu)}$. Such $\bar{k}_\delta$ will be
proved to exist. Then we connect $\|e_{k_\delta}\|$ and
$\|e_{\bar{k}_\delta}\|$ by establishing the inequality
\begin{equation}\label{5.5.5}
\|e_{k_\delta}\|\lesssim
\|e_{\bar{k}_\delta}\|+\frac{1}{\sqrt{\alpha_{\bar{k}_\delta}}}
\left(\|F(x_{k_\delta})-y\|+\delta\right).
\end{equation}
The right hand side can be easily estimated by the desired bound.

$\bullet$ In order to establish (\ref{5.5.5}), we need to
establish the preliminary convergence rate estimate
$\|e_{k_\delta}^\delta\|\lesssim \|u\|^{1/2}\delta^{1/2}$ when
$x_0-x^\dag=F'(x^\dag)^* u$ for some
$u\in {\mathcal N}(F'(x^\dag)^*)^\perp \subset Y$.\\

Therefore, in order to complete the proof of Theorem \ref{T4.1},
we need to establish various estimates.

\subsection{\bf  A first result on convergence rates}

In this subsection we will derive the convergence rate
$\|e_{k_\delta}^\delta\|\lesssim \|u\|^{1/2} \delta^{1/2}$ under
the source condition
\begin{equation}\label{3.00.1}
x_0-x^\dag=F'(x^\dag)^* u, \quad u\in {\mathcal N}(F'(x^\dag)^*)^\perp.
\end{equation}
To this end, we introduce $\tilde{k}_\delta$ to be the first
integer such that
\begin{equation}\label{2.5}
\alpha_{\tilde{k}_\delta} \le \frac{\delta}{\gamma_0
\|u\|}<\alpha_k, \qquad 0\le k<\tilde{k}_\delta,
\end{equation}
where $\gamma_0$ is a number satisfying $\gamma_0>c_0 r/(\tau-1)$,
and $c_0$ is the constant from Assumption \ref{A2.1} (a). Because
of (\ref{4}), such $\tilde{k}_\delta$ is well-defined.

\begin{theorem}\label{T2.1}
Let $\{g_\alpha\}$ and $\{\alpha_k\}$ satisfy Assumption
\ref{A2.1}(a), Assumption \ref{A2.2}, (\ref{3.40}) and (\ref{4}),
and let $F$ satisfy (\ref{2.1}), (\ref{2.3}) and Assumption
\ref{Lip} with $\rho>4\|x_0-x^\dag\|$. Let $\{x_k^\delta\}$ be
defined by (\ref{3}) and let $k_\delta$ be determined by the
discrepancy principle (\ref{2.4}) with $\tau>1$. If $x_0-x^\dag$
satisfies (\ref{3.00.1}) and if $L\|u\|$ is sufficiently small,
then

\begin{enumerate}
\item[(i)]  For all $0\le k\le \tilde{k}_\delta$ there hold
\begin{equation}\label{2.6}
x_k^\delta\in B_\rho(x^\dag) \qquad \mbox{and} \qquad
\|e_k^\delta\|\le 2(c_3+c_4\gamma_0) r^{1/2} \alpha_k^{1/2} \|u\|.
\end{equation}

\item[(ii)] $k_\delta\le \tilde{k}_\delta$, i.e. the discrepancy
principle (\ref{2.4}) is well-defined.

\item[(iii)] There exists a generic constant $C>0$ such that
$$
\|e_{k_\delta}^\delta\|\le C\|u\|^{1/2}\delta^{1/2}.
$$
\end{enumerate}
\end{theorem}

\begin{proof}
We first prove (i). Note that $\rho>4\|x_0-x^\dag\|$, it follows
from (\ref{3.00.1}) and (\ref{2.3}) that (\ref{2.6}) is trivial
for $k=0$. Now for any fixed integer $0<l\le \tilde{k}_\delta$, we
assume that (\ref{2.6}) is true for all $0\le k<l$. It follows
from the definition (\ref{3}) of $\{x_k^\delta\}$ that
\begin{equation}\label{2.7}
e_{k+1}^\delta=r_{\alpha_k} (\A_k^\delta)e_0
-g_{\alpha_k}(\A_k^\delta) F'(x_k^\delta)^*
\left(F(x_k^\delta)-y^\delta-F'(x_k^\delta) e_k^\delta\right).
\end{equation}
Using (\ref{3.00.1}), Assumption \ref{Lip}, Assumption
\ref{A2.1}(a), (\ref{A2.1.2}) and (\ref{1.3}) we obtain
\begin{align*}
\|e_{k+1}^\delta\| &\le  \|r_{\alpha_k}(\A_k^\delta)
F'(x_k^\delta)^* u\|
+\|r_{\alpha_k}(\A_k^\delta)[F'(x^\dag)^*-F'(x_k^\delta)^*] u\|\\
&\quad +c_4 \alpha_k^{-1/2}
\|F(x_k^\delta)-y^\delta-F'(x_k^\delta)
e_k^\delta \|\\
& \le c_3\alpha_k^{1/2}\|u\|+ L\|u\|\|e_k^\delta\| +
\frac{1}{2}c_4 L\|e_k^\delta\|^2\alpha_k^{-1/2} +c_4\delta
\alpha_k^{-1/2}.
\end{align*}
Note that $\delta \alpha_k^{-1}\le \gamma_0 \|u\|$ for $0\le
k<\tilde{k}_\delta$. Note also that $\alpha_k\le r\alpha_{k+1}$ by
(\ref{4}). Therefore, by using (\ref{2.6}) with $k=l-1$, we obtain
\begin{align*}
\|e_l^\delta\|&\le r^{1/2}\alpha_l^{1/2}\left[
(c_3+c_4\gamma_0)\|u\| + L\|u\|\frac{\|e_{l-1}^\delta\|}
{\sqrt{\alpha_{l-1}}} +\frac{1}{2} c_4 L
\left(\frac{\|e_{l-1}^\delta\|}
{\sqrt{\alpha_{l-1}}}\right)^2\right] \\
&\le 2(c_3+c_4 \gamma_0) r^{1/2} \alpha_l^{1/2} \|u\|
\end{align*}
if  $L\|u\|$ is so small that
\begin{equation}\label{s1}
2\left(r^{1/2} +(c_3+c_4\gamma_0) c_4 r\right) L\|u\|\le 1.
\end{equation}
By using (\ref{2.7}), (\ref{A2.1.2}), Assumption \ref{Lip},
(\ref{1.3}),  Assumption \ref{A2.1}(a), (\ref{2.6}) with $k=l-1$
and (\ref{s1}), we also obtain
\begin{align*}
\|e_l^\delta\|&\le \|r_{\alpha_{l-1}}(\A_{l-1}^\delta) e_0\|
+c_4\delta \alpha_{l-1}^{-1/2} +\frac{1}{2}c_4
L\|e_{l-1}^\delta\|^2\alpha_{l-1}^{-1/2}\\
&\le  \|e_0\|+ c_4 \gamma_0^{1/2} \|u\|^{1/2}\delta^{1/2}
+(c_3+c_4 \gamma_0) c_4 r^{1/2} L\|u\|\|e_{l-1}^\delta\|\\
&\le  \|e_0\|+ c_4 \gamma_0^{1/2} \|u\|^{1/2}\delta^{1/2}
+\frac{1}{2} \rho
\end{align*}
Therefore, by using $\rho>4\|e_0\|$, we have
$$
\|e_k^\delta\| \le \frac{3}{4}\rho+ c_4 \gamma_0^{1/2}
\|u\|^{1/2}\delta^{1/2}<\rho
$$
if $\delta>0$ is small enough. Thus (\ref{2.6}) is also true for
all $k=l$. As $l\le \tilde{k}_\delta$ has been arbitrary, we have
completed the proof of (i).

Next we prove (ii) by showing that $k_\delta\le \tilde{k}_\delta$.
From (\ref{2.7}) and (\ref{3.00.1}) we have for $0\le
k<\tilde{k}_\delta$ that
\begin{align*}
F'(x^\dag) e_{k+1}^\delta&-y^\delta+y = F'(x_k^\delta)
r_{\alpha_k}(\A_k^\delta)\left[F'(x_k^\delta)^*
+\left(F'(x^\dag)^*-F'(x_k^\delta)^*\right)\right] u\\
&  +\left[F'(x^\dag)-F'(x_k^\delta)\right]
r_{\alpha_k}(\A_k^\delta)
\left[F'(x_k^\delta)^*+\left(F'(x^\dag)^*
-F'(x_k^\delta)^*\right)\right] u\\ &
-\left[F'(x^\dag)-F'(x_k^\delta)\right]
g_{\alpha_k}(\A_k^\delta)F'(x_k^\delta)^*
\left[F(x_k^\delta)-y^\delta-F'(x_k^\delta) e_k^\delta\right]\\
&  -g_{\alpha_k}(\B_k^\delta)\B_k^\delta
\left[F(x_k^\delta)-y-F'(x_k^\delta)e_k^\delta\right]\\
& -r_{\alpha_k}(\B_k^\delta)(y^\delta-y).
\end{align*}
By using Assumption \ref{Lip}, Assumption \ref{A2.1}(a),
(\ref{A2.1.2}), (\ref{1.3}) and (\ref{2.6}), and noting that
$\delta/\alpha_k\le \gamma_0\|u\|$, we obtain
\begin{align*}
\|F'(x^\dag) e_{k+1}^\delta-y^\delta+y\| &\le  \delta+ c_0
\alpha_k \|u\|+ 2c_3 L\|u\| \alpha_k^{1/2}\|e_k^\delta\|
+ L^2\|u\|\|e_k^\delta\|^2\\
&\quad  + c_4 L\|e_k^\delta\|\delta\alpha_k^{-1/2}
 + \frac{1}{2} c_4 L^2 \alpha_k^{-1/2} \|e_k^\delta\|^3
 + \frac{1}{2} L\|e_k^\delta\|^2 \\
&\le \delta+ \left(c_0 +\varepsilon_1\right) \alpha_k\|u\|,
\end{align*}
where
\begin{align*}
\varepsilon_1=& \left[2r^{1/2}(c_3+c_4 \gamma_0) (2c_3+c_4
\gamma_0)+2(c_3+c_4\gamma_0)^2 r\right] L\|u\|\\
&+ 4\left[(c_3+c_4 \gamma_0)^2 r+(c_3+c_4 \gamma_0)^3 c_4
r^{3/2}\right] L^2 \|u\|^2.
\end{align*}
From (\ref{1.3}), (\ref{3.00.1}) and (\ref{2.3}) we have
$\|F'(x^\dag)e_0-y^\delta+y\|\le \delta+\|\A u\|\le
\delta+c_0\alpha_0 \|u\|$. Thus, by using (\ref{4}),
$$
\|F'(x^\dag)e_k^\delta-y^\delta+y\| \le \delta+
r\left(c_0+\varepsilon_1\right) \alpha_k \|u\|, \quad 0\le k\le
\tilde{k}_\delta.
$$
Consequently
\begin{align*}
\|F(x_{\tilde{k}_\delta}^\delta)-y^\delta\| &\le
\|F'(x^\dag)e_{\tilde{k}_\delta}^\delta-y^\delta+y\|
+\|F(x_{\tilde{k}_\delta}^\delta)-y -F'(x^\dag)
e_{\tilde{k}_\delta}^\delta\|\\
&\le \delta +r\left(c_0+\varepsilon_1\right)
\alpha_{\tilde{k}_\delta}\|u\|
+\frac{1}{2}L\|e_{\tilde{k}_\delta}^\delta\|^2\\
&\le \delta +r\left(c_0+\varepsilon_1+2(c_3+c_4\gamma_0)^2
rL\|u\|\right) \alpha_{\tilde{k}_\delta}\|u\|\\
&\le \delta +r\left(c_0+\varepsilon_1+2(c_3+c_4 \gamma_0)^2 r
L\|u\|\right) \gamma_0^{-1}\delta\\
&\le\tau \delta
\end{align*}
if $L\|u\|$ is so small that
$$
\varepsilon_1+2(c_3+c_4\gamma_0)^2 r L\|u\|\le \frac{(\tau-1)
\gamma_0-c_0 r}{r}.
$$
By the definition of $k_\delta$, it follows that $k_\delta\le
\tilde{k}_\delta$.

Finally we are in a position to derive the convergence rate in
(iii). If $k_\delta=0$, then, by the definition of $k_\delta$, we
have $\|F(x_0)-y^\delta\|\le \tau \delta$. This together with
Assumption \ref{Lip} and (\ref{1.3}) gives
$$
\|F'(x^\dag) e_0\|\le \|F(x_0)-y-F'(x^\dag) e_0\|+\|F(x_0)-y\| \le
\frac{1}{2}L\|e_0\|^2 +(\tau+1)\delta.
$$
Thus, by using (\ref{3.00.1}), we have
\begin{align*}
\|e_0\|&=(e_0, F'(x^\dag)^*u)^{1/2}=(F'(x^\dag) e_0, u)^{1/2}
\le \|F'(x^\dag) e_0\|^{1/2} \|u\|^{1/2}\\
&\le \sqrt{\frac{1}{2} L\|u\|} \|e_0\|+ \sqrt{\tau+1}
\|u\|^{1/2}\delta^{1/2}.
\end{align*}
By assuming that $L\|u\|\le 1$, we obtain
$\|e_{k_\delta}^\delta\|=\|e_0\|\lesssim \|u\|^{1/2}\delta^{1/2}$.

Therefore we will assume $k_\delta>0$ in the following argument.
It follows from (\ref{2.7}), (\ref{A2.1.2}), Assumption \ref{Lip}
and (\ref{2.6}) that for $0\le k<\tilde{k}_\delta$
\begin{align}\label{2.11}
\|e_{k+1}^\delta\|&\le \|r_{\alpha_k}(\A_k^\delta) e_0\|+ c_4
\delta \alpha_k^{-1/2}
+\frac{1}{2}c_4 L\|e_k^\delta\|^2\alpha_k^{-1/2} \nonumber\\
&\le \|r_{\alpha_k}(\A_k^\delta) e_0\|+c_4 (\gamma_0
\|u\|\delta)^{1/2} +(c_3+c_4 \gamma_0) c_4 r^{1/2}
L\|u\|\|e_k^\delta\|.
\end{align}
By (\ref{3.00.1}), (\ref{3.40}) in Assumption \ref{A2.3}, and
Assumption \ref{Lip} we have
\begin{align}\label{2.12}
\|r_{\alpha_k}(\A_k^\delta) e_0-r_{\alpha_k}(\A) e_0\|
&=\|[r_{\alpha_k}(\A_k^\delta)-r_{\alpha_k}(\A)] F'(x^\dag)^* u\| \nonumber\\
&\le c_6 \|u\|\|F'(x_k^\delta)-F'(x^\dag)\|\nonumber \\
&\le c_6 L\|u\|\|e_k^\delta\|.
\end{align}
Thus
\begin{align}\label{2.13}
\|e_{k+1}^\delta\|&\le  \|r_{\alpha_k}(\A) e_0\|+c_4 (\gamma_0
\|u\|\delta)^{1/2} +\left(c_6+ (c_3+c_4\gamma_0) c_4
r^{1/2}\right) L\|u\|\|e_k^\delta\|\nonumber\\
&\le \|r_{\alpha_k}(\A) e_0\|+c_4 (\gamma_0 \|u\|\delta)^{1/2}
+\frac{1}{4c_5} \|e_k^\delta\|
\end{align}
if we assume further that
\begin{equation}\label{s3.1}
4c_5\left(c_6+ (c_3+c_4\gamma_0) c_4 r^{1/2}\right)  L\|u\|\le 1.
\end{equation}
Note that (\ref{2.3}) and the choice of $\beta_0$ imply
$\|r_{\alpha_0}(\A) e_0\| \ge \frac{3}{4}\|e_0\|$. Thus, with the
help of (\ref{2.31}), by induction we can conclude from
(\ref{2.13}) that
$$
\|e_k^\delta\|\le \frac{4}{3} c_5 \|r_{\alpha_k}(\A) e_0\|+ C
\|u\|^{1/2}\delta^{1/2}, \quad 0\le k\le \tilde{k}_\delta.
$$
This together with (\ref{2.12}) and (\ref{s3.1}) implies
\begin{equation}\label{2.14}
\|e_k^\delta\|\le 2c_5 \|r_{\alpha_k} (\A_k^\delta) e_0\|+ C
\|u\|^{1/2} \delta^{1/2}, \quad  0\le k\le \tilde{k}_\delta.
\end{equation}
The combination of (\ref{2.11}), (\ref{2.14}) and (\ref{s3.1})
gives
\begin{equation}\label{2.11.5}
\|e_{k+1}^\delta\|\le \frac{3}{2}\|r_{\alpha_k}(\A_k^\delta)
e_0\|+ C\|u\|^{1/2} \delta^{1/2}, \quad 0\le k<\tilde{k}_\delta.
\end{equation}
We need to estimate $\|r_{\alpha_k}(\A_k^\delta) e_0\|$. By
(\ref{3.00.1}), Assumption \ref{A2.1}(a) and Assumption \ref{Lip}
we have
\begin{align*}
\|r_{\alpha_k}(\A_k^\delta) e_0\|^2 &=\left(
r_{\alpha_k}(\A_k^\delta) e_0, r_{\alpha_k}(\A_k^\delta)
F'(x^\dag)^* u\right)\\
&=\left(r_{\alpha_k}(\A_k^\delta) e_0, r_{\alpha_k}(\A_k^\delta)
\left[F'(x_k^\delta)^*
+\left(F'(x^\dag)^*-F'(x_k^\delta)^*\right)\right] u \right)\\
&\le \|F'(x_k^\delta) r_{\alpha_k}(\A_k^\delta) e_0\|\|u\|
+L\|u\|\|e_k^\delta\|\|r_{\alpha_k}(\A_k^\delta) e_0\|.
\end{align*}
Thus
$$
\|r_{\alpha_k}(\A_k^\delta) e_0\| \le \|F'(x_k^\delta)
r_{\alpha_k} (\A_k^\delta) e_0\|^{1/2} \|u\|^{1/2}
+L\|u\|\|e_k^\delta\|.
$$
With the help of (\ref{2.7}), (\ref{1.3}), Assumption
\ref{A2.1}(a) and Assumption \ref{Lip} we have
\begin{align*}
\|F'(x_k^\delta) r_{\alpha_k}(\A_k^\delta) e_0\| &\le
\|F'(x_k^\delta) e_{k+1}^\delta\|
+\|g_{\alpha_k}(\B_k^\delta)\B_k^\delta
\left(F(x_k^\delta)-y^\delta-F'(x_k^\delta) e_k^\delta\right)\|\\
& \le  \|F(x_{k+1}^\delta)-y^\delta\|+ 2\delta
+\|F(x_{k+1}^\delta)-y-F'(x_{k+1}^\delta)e_{k+1}^\delta\|\\
& \quad +\|[F'(x_{k+1}^\delta)-F'(x_k^\delta)] e_{k+1}^\delta\|
+\|F(x_k^\delta)-y-F'(x_k^\delta) e_k^\delta\|\\
&\le \|F(x_{k+1}^\delta)-y^\delta\|+ 2\delta+L\|e_k^\delta\|^2+
2L\|e_{k+1}^\delta\|^2.
\end{align*}
Therefore
\begin{align*}
\|r_{\alpha_k}(\A_k^\delta) e_0\| & \le  \|u\|^{1/2}
\|F(x_{k+1}^\delta)-y^\delta\|^{1/2}
+\sqrt{2}\|u\|^{1/2}\delta^{1/2}+\sqrt{2L\|u\|}\|e_{k+1}^\delta\|\\
&\quad + \left(L\|u\|+\sqrt{L\|u\|}\right) \|e_k^\delta\|.
\end{align*}
Combining this with (\ref{2.14}) and (\ref{2.11.5}) yields
\begin{align*}
\|r_{\alpha_k}(\A_k^\delta) e_0\| & \le
\|u\|^{1/2}\|F(x_{k+1}^\delta)-y^\delta\|^{1/2} +
C\|u\|^{1/2}\delta^{1/2}\\
&\quad +\frac{1}{2}\left[\left(3\sqrt{2}+4c_5\right)\sqrt{L\|u\|}
+4c_5L\|u\|\right] \|r_{\alpha_k}(\A_k^\delta) e_0\|.
\end{align*}
Thus, if
$$
\left(3\sqrt{2}+4 c_5\right)\sqrt{L\|u\|} +4 c_5L\|u\|\le 1,
$$
we then obtain
$$
\|r_{\alpha_k}(\A_k^\delta) e_0\| \lesssim
\|u\|^{1/2}\|F(x_{k+1}^\delta)-y^\delta\|^{1/2}
+\|u\|^{1/2}\delta^{1/2}.
$$
This together with (\ref{2.11.5}) gives
$$
\|e_k^\delta\| \lesssim
\|u\|^{1/2}\|F(x_k^\delta)-y^\delta\|^{1/2}
+\|u\|^{1/2}\delta^{1/2}
$$
for all $0<k\le \tilde{k}_\delta$. Consequently, we may set
$k=k_\delta$ in the above inequality and use the definition of
$k_\delta$ to obtain $\|e_{k_\delta}^\delta\|\lesssim
\|u\|^{1/2}\delta^{1/2}$. \hfill $\Box$
\end{proof}

\subsection{\bf Stability estimates}

In this subsection we will consider the stability of the method
(\ref{3}) by deriving some useful estimates on
$\|x_k^\delta-x_k\|$, where $\{x_k\}$ is defined by (\ref{1.8}).
It is easy to see that
\begin{equation}\label{3.9}
e_{k+1}=r_{\alpha_k}(\A_k)e_0 -g_{\alpha_k}
(\A_k)F'(x_k)^*\left(F(x_k)-y-F'(x_k)e_k\right).
\end{equation}
We will prove some important estimates on $\{x_k\}$ in Lemma
\ref{L3.1} in the next subsection. In particular, we will show
that, under the conditions in Theorem \ref{T2.1},
\begin{equation}\label{3.2.1}
x_k\in B_\rho(x^\dag) \qquad \mbox{and} \qquad \|e_k\|\le 2 c_3
r^{1/2} \alpha_k^{1/2}\|u\|
\end{equation}
for all $k\ge 0$ provided $L\|u\|$ is sufficiently small.

\begin{lemma}\label{L3.2}
Let all the conditions in Theorem \ref{T2.1} and Assumption
\ref{A2.3} hold. If $L\|u\|$ is sufficiently small, then for all
$0\le k\le \tilde{k}_\delta$ there hold
\begin{equation}\label{3.51}
\|x_k^\delta-x_k\|\le 2c_4  \frac{\delta}{\sqrt{\alpha_k}}
\end{equation}
and
\begin{equation}\label{3.52}
\|F(x_k^\delta)-F(x_k)-y^\delta+y\|\le (1+\varepsilon_2)\delta,
\end{equation}
where
\begin{align*}
\varepsilon_2 &:=2c_4 \left((c_6+ rc_4 \gamma_0)+ (4c_3+3c_4
\gamma_0)r^{1/2}+4(c_3+c_4\gamma_0) r\right) L\|u\|\\
&\quad +4 c_3c_4\left(c_6 r^{1/2}+ (c_4+c_6)c_3 r\right) L^2
\|u\|^2.
\end{align*}
\end{lemma}

\begin{proof}
For each $0\le k\le \tilde{k}_\delta$ we set
\begin{equation}
u_k:=F(x_k)-y-F'(x_k) e_k, \qquad  u_k^\delta:=F(x_k^\delta)-y
-F'(x_k^\delta) e_k^\delta.
\end{equation}
It then follows from (\ref{2.7}) and (\ref{3.9}) that
\begin{equation}\label{3.53}
x_{k+1}^\delta-x_{k+1}=I_1+I_2+I_3+I_4,
\end{equation}
where
\begin{align*}
I_1&:=\left[r_{\alpha_k}(\A_k^\delta)-r_{\alpha_k}(\A_k)\right] e_0,\\
I_2&:=g_{\alpha_k}(\A_k^\delta) F'(x_k^\delta)^* (y^\delta-y),\\
I_3&:=\left[g_{\alpha_k}(\A_k) F'(x_k)^*-g_{\alpha_k}(\A_k^\delta)
F'(x_k^\delta)^*\right] u_k,\\
I_4&:=g_{\alpha_k}(\A_k^\delta) F'(x_k^\delta)^* (u_k-u_k^\delta).
\end{align*}
By using (\ref{3.00.1}), (\ref{3.41}), (\ref{3.40}), Assumption
\ref{Lip} and (\ref{3.2.1}) we have
\begin{align*}
\|I_1\| & \le \|r_{\alpha_k}(\A_k^\delta)
-r_{\alpha_k}(\A_k)\|\|F'(x^\dag)^*-F'(x_k)^*\|\|u\|\\
&\quad +\|[r_{\alpha_k}(\A_k^\delta)-r_{\alpha_k}(\A_k)] F'(x_k)^* u\|\\
& \le  c_6 L^2\|u\|\|e_k\|\|x_k^\delta-x_k\|\alpha_k^{-1/2}
+c_6 L\|u\|\|x_k^\delta-x_k\|\\
& \le c_6\left(L\|u\|+2c_3r^{1/2} L^2 \|u\|^2\right)
\|x_k^\delta-x_k\|.
\end{align*}
With the help of (\ref{A2.1.2}) and (\ref{1.3}) we have
$$
\|I_2\|\le c_4 \frac{\delta}{\sqrt{\alpha_k}}.
$$
By applying Assumption \ref{A2.1}(a), (\ref{3.43}), Assumption
\ref{Lip} and (\ref{3.2.1}) we can estimate $I_3$ as
\begin{align*}
\|I_3\|&\le \|g_{\alpha_k}(\A_k)[F'(x_k^\delta)^*-F'(x_k)^*] u_k\|
+ \|[g_{\alpha_k}(\A_k)-g_{\alpha_k}(\A_k^\delta)] F'(x_k^\delta)^* u_k\|\\
&\le (c_1+c_6) L\|u_k\|\|x_k^\delta-x_k\| \alpha_k^{-1}
\le \frac{1}{2}(c_1+c_6) L^2\|e_k\|^2 \|x_k^\delta-x_k\| \alpha_k^{-1}\\
&\le 2(c_1+c_6)c_3^2 r  L^2\|u\|^2 \|x_k^\delta-x_k\|.
\end{align*}
For the term $I_4$, we have from (\ref{A2.1.2}) that
$$
\|I_4\|\le \frac{c_4}{\sqrt{\alpha_k}}\|u_k^\delta-u_k\|.
$$
By using Assumption \ref{Lip}, (\ref{2.6}) and (\ref{3.2.1}) one
can see
\begin{align}\label{3.54}
\|u_k-u_k^\delta\|&\le
\|F(x_k^\delta)-F(x_k)-F'(x_k)(x_k^\delta-x_k)\|
+\|[F'(x_k^\delta)-F'(x_k)]e_k^\delta\|\nonumber\\
&\le \frac{1}{2}
L\|x_k^\delta-x_k\|^2+L\|e_k^\delta\|\|x_k^\delta-x_k\| \le
\frac{1}{2} L\left(3\|e_k^\delta\|+\|e_k\|\right)
\|x_k^\delta-x_k\|\nonumber\\
&\le \left(4c_3+3 c_4 \gamma_0\right) r^{1/2} \alpha_k^{1/2}
L\|u\| \|x_k^\delta-x_k\|.
\end{align}
Therefore
$$
\|I_4\|\le \left(4c_3+3c_4\gamma_0\right) c_4 r^{1/2}  L\|u\|
\|x_k^\delta-x_k\|.
$$
Thus, if $L\|u\|$ is so small that
$$
\left(c_6+(4c_3+3c_4 \gamma_0)c_4 r^{1/2}\right) L\|u\|
+2\left(c_3c_6 r^{1/2}+c_3^2 (c_1+c_6)r\right)L^2 \|u\|^2\le
\frac{1}{2},
$$
then the combination of the above estimates on $I_1$, $I_2$, $I_3$
and $I_4$ gives for $0\le k<\tilde{k}_\delta$ that
$$
\|x_{k+1}^\delta-x_{k+1}\| \le  c_4
\frac{\delta}{\sqrt{\alpha_k}}+  \frac{1}{2}\|x_k^\delta-x_k\|.
$$
This implies (\ref{3.51}) immediately.

Next we prove (\ref{3.52}). We have from (\ref{3.53}) that
\begin{equation}\label{4.27.8}
F'(x_k^\delta)(x_{k+1}^\delta-x_{k+1})-y^\delta+y
=F'(x_k^\delta)\left(I_1+I_2+I_3+I_4\right)-y^\delta+y.
\end{equation}
From (\ref{3.00.1}), (\ref{3.40}), (\ref{3.42}), Assumption
\ref{Lip}, (\ref{3.2.1}) and (\ref{3.51}) it follows that
\begin{align*}
\|F'(x_k^\delta) I_1\| & \le
\|F'(x_k^\delta)[r_{\alpha_k}(\A_k^\delta)-r_{\alpha_k}(\A_k)]
[F'(x^\dag)^*-F'(x_k)^*] u\|\\
& \quad +\|F'(x_k^\delta)[r_{\alpha_k}(\A_k^\delta)-r_{\alpha_k}(\A_k)]
F'(x_k)^*u\|\\
& \le c_6 L^2\|u\|\|e_k\|\|x_k^\delta-x_k\| + c_6
L\|u\|\alpha_k^{1/2}\|x_k^\delta-x_k\|\\
& \le \left(2c_4c_6 L\|u\|+4c_3 c_4 c_6 r^{1/2} L^2 \|u\|^2\right)
\delta.
\end{align*}
By using Assumption \ref{A2.1}(a) and (\ref{1.3}) it is easy to
see
\begin{equation}\label{4.27.9}
\|F'(x_k^\delta) I_2-y^\delta+y\|
=\|r_{\alpha_k}(\B_k^\delta)(y^\delta-y)\|\le \delta.
\end{equation}
In order to estimate $F'(x_k^\delta) I_3$, we note that
\begin{equation}\label{4.28.3}
F'(x_k^\delta) I_3=\left[F'(x_k^\delta)-F'(x_k)\right]
g_{\alpha_k}(\A_k)F'(x_k)^* u_k
+\left[r_{\alpha_k}(\B_k^\delta)-r_{\alpha_k}(\B_k)\right] u_k.
\end{equation}
Thus, it follows from (\ref{A2.1.2}), Assumption \ref{Lip},
(\ref{3.41}), (\ref{3.2.1}) and (\ref{3.51}) that
\begin{align*}
\|F'(x_k^\delta) I_3\| & \le \|\left[F'(x_k^\delta)-F'(x_k)\right]
g_{\alpha_k}(\A_k)F'(x_k)^* u_k\|\\
& \quad  +\|\left[r_{\alpha_k}(\B_k^\delta)-r_{\alpha_k}(\B_k)\right] u_k\|\\
& \le (c_4+c_6) \alpha_k^{-1/2} L\|x_k^\delta-x_k\|\|u_k\|\\
& \le \frac{1}{2} (c_4+c_6) \alpha_k^{-1/2} L^2 \|e_k\|^2 \|x_k^\delta-x_k\|\\
& \le 4 (c_4+c_6)c_3^2 c_4 r L^2 \|u\|^2 \delta.
\end{align*}
For the term $F'(x_k^\delta)I_4$ we have from Assumption
\ref{A2.1}(a), (\ref{3.54}) and (\ref{3.51}) that
\begin{align*}
\|F'(x_k^\delta) I_4\| \le  \|u_k-u_k^\delta\|\le 2 (4c_3+3c_4
\gamma_0) c_4 r^{1/2} L\|u\|\delta.
\end{align*}
Combining the above estimates, we therefore obtain
$$
\|F'(x_k^\delta)(x_{k+1}^\delta-x_{k+1})-y^\delta+y\| \le
(1+\varepsilon_3)\delta, \qquad 0\le k< \tilde{k}_\delta,
$$
where
\begin{align*}
\varepsilon_3:=& 2c_4 \left(c_6+ (4c_3+3c_4 \gamma_0) r^{1/2}
\right) L\|u\|+ 4 c_3c_4\left(c_6 r^{1/2}+ (c_4+c_6) c_3 r\right)
L^2 \|u\|^2.
\end{align*}
This together with Assumption \ref{Lip}, (\ref{2.6}),
(\ref{3.51}) and (\ref{4}) implies for $0\le k<\tilde{k}_\delta$ that
\begin{align*}
\|F'(x_{k+1}^\delta)&(x_{k+1}^\delta-x_{k+1})-y^\delta+y\|\\
&\le  \|F'(x_k^\delta)(x_{k+1}^\delta-x_{k+1})-y^\delta+y\|
+L\|x_{k+1}^\delta-x_k^\delta\|\|x_{k+1}^\delta-x_{k+1}\|\\
&\le  (1+\varepsilon_3 )\delta+ 2 c_4
L(\|e_{k+1}^\delta\|+\|e_k^\delta\|)\frac{\delta}{\sqrt{\alpha_{k+1}}}\\
& \le  (1+\varepsilon_4)\delta,
\end{align*}
where
$$
\varepsilon_4:=\varepsilon_3+8(c_3+c_4\gamma_0)c_4 r L\|u\|.
$$
Thus
$$
\|F'(x_k^\delta)(x_k^\delta-x_k)-y^\delta+y\| \le
(1+\varepsilon_4)\delta, \qquad 0\le k\le \tilde{k}_\delta.
$$
Therefore, noting that $\delta/\alpha_k \le r\gamma_0 \|u\|$ for
$0\le k\le  \tilde{k}_\delta$, we have
\begin{align*}
\|F(x_k^\delta)-F(x_k)-y^\delta+y\| & \le
\|F(x_k^\delta)-F(x_k)-F'(x_k^\delta)(x_k^\delta-x_k)\|\\
& \quad +\|F'(x_k^\delta)(x_k^\delta-x_k)-y^\delta+y\|\\
& \le \frac{1}{2} L\|x_k^\delta-x_k\|^2+(1+\varepsilon_4)\delta\\
& \le  2c_4^2 L\frac{\delta}{\alpha_k} \delta+(1+\varepsilon_4)\delta\\
& \le (1+\varepsilon_4+ 2rc_4^2 \gamma_0 L\|u\|)\delta.
\end{align*}
The proof of (\ref{3.52}) is thus complete. \hfill $\Box$
\end{proof}

\subsection{\bf Some estimates on noise-free iterations}

\begin{lemma}\label{L3.1}
Let all the conditions in Theorem \ref{T2.1} be fulfilled. If
$L\|u\|$ is sufficiently small, then for all $k\ge 0$ we have
\begin{equation}\label{3.2}
x_k\in B_\rho(x^\dag) \qquad \mbox{and}  \qquad \|e_k\|\le 2 c_3
r^{1/2} \alpha_k^{1/2}\|u\|.
\end{equation}
If,  in addition, Assumption \ref{A2.1}(b) is satisfied, then
\begin{equation}\label{3.3}
\frac{2}{3}\|r_{\alpha_k}(\A) e_0\| \le \|e_k\|\le \frac{4}{3} c_5
\|r_{\alpha_k}(\A) e_0\|
\end{equation}
and
\begin{equation}\label{4.27.1}
\frac{1}{2c_5}\|e_k\|\le \|e_{k+1}\|\le 2 \|e_k\|.
\end{equation}
\end{lemma}

\begin{proof}
By using (\ref{3.00.1}), (\ref{A2.1.2}), (\ref{3.40}) and
Assumption \ref{Lip}, we have from (\ref{3.9}) that
\begin{align}\label{3.21}
 \|e_{k+1}-r_{\alpha_k}(\A) e_0\|
& \le \|[r_{\alpha_k}(\A_k)-r_{\alpha_k}(\A)] F'(x^\dag)^* u\|
+\frac{c_4}{\sqrt{\alpha_k}}\|F(x_k)-y-F'(x_k) e_k\|\nonumber\\
& \le c_6 L\|u\|\|e_k\|+\frac{c_4}{2\sqrt{\alpha_k}}L\|e_k\|^2.
\end{align}
Since (\ref{A2.1.2}) and (\ref{3.00.1}) imply $\|r_{\alpha_k}(\A)
e_0\|\le c_3 \alpha_k^{1/2}\|u\|$, we have
$$
\|e_{k+1}\|\le c_3 \alpha_k^{1/2}\|u\|+c_6 L\|u\|\|e_k\|
+\frac{c_4}{2\sqrt{\alpha_k}}L\|e_k\|^2.
$$
Note that (\ref{3.00.1}) and (\ref{2.3}) imply $\|e_0\|\le c_3
\alpha_0^{1/2} \|u\|$. By induction one can conclude the assertion
(\ref{3.2}) if $L\|u\|$ is so small that $2(c_6 r^{1/2} +c_3 c_4
r) L\|u\|\le 1$.

If we assume further that
\begin{equation}\label{4.27.2}
5c_5\left(c_6+c_3c_4 r^{1/2} \right)L\|u\|\le 1,
\end{equation}
the combination of (\ref{3.21}) and (\ref{3.2}) gives
\begin{equation}\label{3.22}
\|e_{k+1}-r_{\alpha_k}(\A) e_0\|\le \left(c_6+c_3 c_4 r^{1/2}
\right)  L\|u\|\|e_k\| \le \frac{1}{5c_5}\|e_k\|.
\end{equation}
Note that Assumption \ref{A2.1}(b) and $\alpha_k\le \alpha_{k-1}$
imply $\|r_{\alpha_k}(\A)e_0\|\le \|r_{\alpha_{k-1}}(\A)e_0\|$.
Note also that Assumption \ref{A2.1}(a) and (\ref{2.3}) imply
(\ref{3.3}) with $k=0$. Thus, from (\ref{3.22}) and (\ref{2.31})
we can conclude (\ref{3.3}) by an induction argument.
(\ref{4.27.1}) is an immediate consequence of (\ref{3.22}) and
(\ref{3.3}). \hfill $\Box$
\end{proof}

\begin{lemma}\label{L3.3}
Let all the conditions in Lemma \ref{L3.2} and Assumption
\ref{A2.1}(c) hold. If $k_\delta>0$ and $L\|u\|$ is sufficiently
small, then for all $k\ge k_\delta$ we have
\begin{equation}\label{3.4}
\|e_{k_\delta}\|\lesssim \|e_k\| +\frac{1}{\sqrt{\alpha_k}}
\left(\|F(x_{k_\delta})-y\|+\delta\right).
\end{equation}
\end{lemma}

\begin{proof}
It follows from (\ref{3.9}) that
\begin{align}\label{12.10}
x_{k_\delta}-x_k & =[r_{\alpha_{k_\delta-1}}(\A)-r_{\alpha_{k-1}}(\A)]e_0
+[r_{\alpha_{k_\delta-1}}(\A_{k_\delta-1})-r_{\alpha_{k_\delta-1}}(\A)]
e_0\nonumber\\
&\quad -\left[r_{\alpha_{k-1}}(\A_{k-1})-r_{\alpha_{k-1}}(\A)\right] e_0\nonumber\\
&\quad -g_{\alpha_{k_\delta-1}}(\A_{k_\delta-1}) F'(x_{k_\delta-1})^*
\left[F(x_{k_\delta-1})-y-F'(x_{k_\delta-1})e_{k_\delta-1}\right]\nonumber\\
&\quad +g_{\alpha_{k-1}}(\A_{k-1})
F'(x_{k-1})^*\left[F(x_{k-1})-y-F'(x_{k-1})e_{k-1}\right].
\end{align}
Thus, by using (\ref{3.00.1}), (\ref{3.40}), Assumption \ref{Lip},
(\ref{A2.1.2}), (\ref{3.2}) and (\ref{4.27.2}), we have
\begin{align}\label{3.10}
 \|x_{k_\delta}-x_k\|
&\le  \|[ r_{\alpha_{k_\delta-1}}(\A)-r_{\alpha_{k-1}}(\A)]e_0\|
+c_6 L\|u\|\left(\|e_{k-1}\|+\|e_{k_\delta-1}\|\right)\nonumber\\
&\quad
+\frac{c_4}{2\sqrt{\alpha_{k_\delta-1}}}L\|e_{k_\delta-1}\|^2
+\frac{c_4}{2\sqrt{\alpha_{k-1}}}L\|e_{k-1}\|^2\nonumber\\
& \le \|[ r_{\alpha_{k_\delta-1}}(\A)-r_{\alpha_{k-1}}(\A)]e_0\|
+\frac{1}{5c_5} \left(\|e_{k-1}\|+\|e_{k_\delta-1}\|\right).
\end{align}
Since $k\ge k_\delta$, we have $\alpha_{k-1}\le
\alpha_{k_\delta-1}$. Since Assumption \ref{A2.1}(b) and (c) hold,
we may apply Lemma \ref{L10.1} with $x=e_0$,
$\bar{x}=e_{k_\delta}$, $\alpha=\alpha_{k-1}$,
$\beta=\alpha_{k_\delta-1}$ and $A=F'(x^\dag)$ to obtain
$$
\|[ r_{\alpha_{k_\delta-1}}(\A)-r_{\alpha_{k-1}}(\A)]e_0\| \le
\|r_{\alpha_{k_\delta-1}}(\A) e_0-e_{k_\delta}\|
+\frac{c_2}{\sqrt{\alpha_{k-1}}} \|F'(x^\dag) e_{k_\delta}\|.
$$
Note that (\ref{3.22}) implies
$$
\|e_{k_\delta}-r_{\alpha_{k_\delta-1}}(\A)e_0\| \le
\frac{1}{5c_5}\|e_{k_\delta-1}\|.
$$
Note also that Assumption \ref{Lip} implies
$$
\|F'(x^\dag) e_{k_\delta}\|\le \|F(x_{k_\delta})-y\|+\frac{1}{2}
L\|e_{k_\delta}\|^2.
$$
Thus
$$
\|[ r_{\alpha_{k_\delta-1}}(\A)-r_{\alpha_{k-1}}(\A)]e_0\| \le
\frac{1}{5c_5} \|e_{k_\delta-1}\| +\frac{C}{\sqrt{\alpha_k}}
\left(\|F(x_{k_\delta})-y\|+L\|e_{k_\delta}\|^2\right).
$$
Since Lemma \ref{L3.2}, Theorem \ref{T2.1} and the fact
$k_\delta\le \tilde{k}_\delta$ imply
$$
\|e_{k_\delta}\| \lesssim \|e_{k_\delta}^\delta\|
+\frac{\delta}{\sqrt{\alpha_{k_\delta}}}\lesssim
\|u\|^{1/2}\delta^{1/2},
$$
we have
$$
\|[ r_{\alpha_{k_\delta-1}}(\A)-r_{\alpha_{k-1}}(\A)]e_0\| \le
\frac{1}{5c_5} \|e_{k_\delta-1}\|+\frac{C}{\sqrt{\alpha_{k}}}
\left(\|F(x_{k_\delta})-y\|+L\|u\|\delta\right).
$$
Combining this with (\ref{3.10}) and using Lemma \ref{L3.1} gives
$$
\|x_{k_\delta}-x_k\|\le \frac{4}{5} \|e_{k_\delta}\|+ C\|e_k\|
+\frac{C}{\sqrt{\alpha_k}}\left(\|F(x_{k_\delta})-y\|+\delta\right).
$$
This completes the proof. \hfill $\Box$
\end{proof}

\subsection{\bf Completion of proof of Theorem \ref{T4.1}}

\begin{lemma}\label{L4.1}
Assume that all the conditions in Lemma \ref{L3.1} are satisfied.
Then
\begin{equation}\label{4.2}
\|F'(x^\dag) e_k\|\lesssim \|r_{\alpha_k}(\A) \A^{1/2} e_0\|+
\alpha_k^{1/2}\|r_{\alpha_k}(\A) e_0\|
\end{equation}
for all $k\ge 0$.
\end{lemma}

\begin{proof}
We first use (\ref{3.9}) to write
\begin{align}\label{4.27.3}
F'(x^\dag) e_{k+1}&=F'(x^\dag) r_{\alpha_k}(\A) e_0 +F'(x^\dag)
\left[r_{\alpha_k}(\A_k)-r_{\alpha_k}(\A)\right] e_0\nonumber\\
&\quad -F'(x^\dag) g_{\alpha_k}(\A_k) F'(x_k)^*\left[F(x_k)-y-F'(x_k)
e_k\right].
\end{align}
Thus, it follows from (\ref{3.00.1}), Assumption \ref{Lip},
Assumption \ref{A2.1}(a), (\ref{3.40}), (\ref{3.42}), (\ref{3.2})
and (\ref{3.3}) that
\begin{align*}
\|F'(x^\dag) e_{k+1}\| &\lesssim \|F'(x^\dag) r_{\alpha_k}(\A)
e_0\|
+L\|e_k\|\|[r_{\alpha_k}(\A_k)-r_{\alpha_k}(\A)]F'(x^\dag)^* u\|\\
&\quad +\|F'(x_k)[r_{\alpha_k}(\A_k)-r_{\alpha_k}(\A)]F'(x^\dag)^* u\|
+(1+L\|e_k\|\alpha_k^{-1/2}) L\|e_k\|^2\\
&\lesssim \|r_{\alpha_k}(\A)\A^{1/2}
e_0\|+L^2\|u\|\|e_k\|^2+\alpha_k^{1/2}L\|u\|\|e_k\|+L\|e_k\|^2\\
&\lesssim  \|r_{\alpha_k}(\A)\A^{1/2} e_0\|+
\alpha_k^{1/2}\|r_{\alpha_k}(\A) e_0\|.
\end{align*}
This together with (\ref{2.31}) and (\ref{4}) implies (\ref{4.2}). \hfill $\Box$
\end{proof}

\begin{lemma}\label{L4.2}
Under the conditions in Lemma \ref{L3.2} and Lemma \ref{L3.1}, if
$\varepsilon_2\le (\tau-1)/2$ then for the $k_\delta$ determined
by (\ref{2.4}) with $\tau>1$ we have
\begin{equation}\label{4.3}
(\tau-1) \delta \lesssim \|r_{\alpha_k}(\A) \A^{1/2} e_0\|+
\alpha_k^{1/2}\|r_{\alpha_k}(\A) e_0\|
\end{equation}
for all $0\le k<k_\delta$,
\end{lemma}

\begin{proof}
By using (\ref{3.52}), Lemma \ref{L3.1} and Lemma \ref{L4.1}, we
have for $0\le k<k_\delta$ that
\begin{align*}
\tau \delta &\le \|F(x_k^\delta)-y^\delta\|
\le \|F(x_k^\delta)-F(x_k)-y^\delta+y\|+\|F(x_k)-y\|\\
& \le (1+\varepsilon_2)\delta+\|F'(x^\dag)e_k\|+\frac{1}{2}L\|e_k\|^2\\
& \le (1+\varepsilon_2)\delta +C \|r_{\alpha_k}(\A) \A^{1/2} e_0\|+
C\alpha_k^{1/2}\|r_{\alpha_k}(\A) e_0\|.
\end{align*}
Since $\tau>1$, by the smallness condition $\varepsilon_2\le
(\tau-1)/2$ on $L\|u\|$ we obtain (\ref{4.3}). \hfill $\Box$
\end{proof}

\begin{proof} {\it of Theorem \ref{T4.1}}.
If $k_\delta=0$, then the definition of $k_\delta$ implies
$\|F(x_0)-y^\delta\|\le \tau \delta$. From Theorem \ref{T2.1} we
know that $\|e_0\|\lesssim \|u\|^{1/2}\delta^{1/2}$. Thus
\begin{align*}
\|F'(x^\dag) e_0\| &\le \|F(x_0)-y-F'(x^\dag)
e_0\|+\|F(x_0)-y^\delta\|+\delta\\
&\le \frac{1}{2}L\|e_0\|^2+(1+\tau)\delta\lesssim \delta.
\end{align*}
Since $e_0=\A^{\nu}\omega$ for some $1/2\le \nu\le \bar{\nu}-1/2$,
we may use the interpolation inequality to obtain
\begin{align*}
\|e_{k_\delta}^\delta\|&=\|e_0\|=\|\A^\nu \omega\|\le
\|\omega\|^{1/(1+2\nu)} \|\A^{1/2+\nu}\omega\|^{2\nu/(1+2\nu)}\\
&=\|\omega\|^{1/(1+2\nu)} \|F'(x^\dag) e_0\|^{2\nu/(1+2\nu)}\\
&\lesssim \|\omega\|^{1/(1+2\nu)} \delta^{2\nu/(1+2\nu)},
\end{align*}
which gives the desired estimate.

Therefore, we may assume that $k_\delta>0$ in the remaining
argument. By using $e_0=\A^\nu \omega$ for some $1/2 \le\nu\le
\bar{\nu}- 1/2$ and Lemma \ref{L4.2} it follows that there exists
a positive constant $C_\nu$ such that
$$
(\tau-1)\delta< C_\nu \alpha_k^{\nu+1/2}\|\omega\|, \qquad 0\le
k<k_\delta.
$$
Now we define the integer $\bar{k}_\delta$ by
$$
\alpha_{\bar{k}_\delta} \le
\left(\frac{(\tau-1)\delta}{C_\nu\|\omega\|}\right)^{2/(1+2\nu)}<\alpha_k,
\qquad 0\le k<\bar{k}_\delta.
$$
Then $k_\delta\le \bar{k}_\delta$. Thus, by using Lemma \ref{L3.2}
and Lemma \ref{L3.3}, we have
$$
\|e_{k_\delta}^\delta\| \lesssim
\|e_{k_\delta}\|+\frac{\delta}{\sqrt{\alpha_{k_\delta}}} \lesssim
\|e_{\bar{k}_\delta}\| +\frac{\|F(x_{k_\delta})-y\|
+\delta}{\sqrt{\alpha_{\bar{k}_\delta}}}
+\frac{\delta}{\sqrt{\alpha_{k_\delta}}}.
$$
Note that Lemma \ref{L3.2} and the definition of $k_\delta$ imply
$$
\|F(x_{k_\delta})-y\|\le \|F(x_{k_\delta}^\delta)-y^\delta\|
+\|F(x_{k_\delta}^\delta)-F(x_{k_\delta})-y^\delta+y\|\lesssim
\delta.
$$
This together with (\ref{3.3}), $k_\delta\le \bar{k}_\delta$ and
$\|r_{\alpha_k}(\A) e_0\| \lesssim \alpha_k^\nu\|\omega\|$ then
gives
\begin{equation}\label{6.18.1}
\|e_{k_\delta}^\delta\|\lesssim
\alpha_{\bar{k}_\delta}^\nu\|\omega\|
+\frac{\delta}{\sqrt{\alpha_{k_\delta}}}
+\frac{\delta}{\sqrt{\alpha_{\bar{k}_\delta}}} \lesssim
\alpha_{\bar{k}_\delta}^\nu\|\omega\| +
\frac{\delta}{\sqrt{\alpha_{\bar{k}_\delta}}}.
\end{equation}
Using the definition of $\bar{k}_\delta$ and (\ref{4}), we
therefore complete the proof. \hfill $\Box$
\end{proof}

\section{\bf Proof of Theorem \ref{T4.5}}
\setcounter{equation}{0}

In this section we will give the proof of Theorem \ref{T4.5}. The
essential idea is similar as in the proof of Theorem \ref{T4.1}.
Thus we need to establish similar results as those used in Section
3. However, since we do not have source representation
$e_0=F'(x^\dag)^*u$ any longer and since $F$ satisfies different
conditions, we must modify the arguments carefully. We will
indicate the essential steps without spelling out all the
necessary smallness conditions on $(K_0+K_1+K_2)\|e_0\|$. We first
introduce the integer $n_\delta$ by
\begin{equation}\label{11.1}
\alpha_{n_\delta}\le \left(\frac{\delta}{\gamma_1\|e_0\|}\right)^2
<\alpha_k, \qquad 0\le k<n_\delta.
\end{equation}
Recall that $\gamma_1$ is a constant satisfying $\gamma_1>c_3
r^{1/2}/(\tau-1)$.

\begin{proof}{\it of Theorem \ref{T4.5}}.  In order to complete the
proof of Theorem \ref{T4.5}, we need to establish various estimates.
We will divide the arguments into several steps.

{\it Step 1}. We will show that for all $0\le k\le n_\delta$
\begin{equation}\label{11.2}
x_k^\delta\in B_\rho(x^\dag), \quad \|e_k^\delta\|\lesssim
\|e_0\|,
\end{equation}
\begin{equation}\label{11.2.5}
\|F'(x^\dag) e_k^\delta\|\lesssim \alpha_k^{1/2} \|e_0\|
\end{equation}
and that $k_\delta\le n_\delta$ for the integer $k_\delta$ defined
by the discrepancy principle (\ref{2.4}) with $\tau>1$.

To see this, we note that, for any $0\le k<n_\delta$ with
$x_k^\delta\in B_\rho(x^\dag)$, (\ref{2.7}) and Assumption
\ref{F1} imply
\begin{align*}
e_{k+1}^\delta&=r_{\alpha_k}(\A_k^\delta) e_0 -\int_0^1
g_{\alpha_k}(\A_k^\delta)\A_k^\delta
\left(R(x_k^\delta-te_k^\delta, x_k^\delta)-I\right)
e_k^\delta dt\\
&\quad +g_{\alpha_k}(\A_k^\delta) F'(x_k^\delta)^*(y^\delta-y).
\end{align*}
Therefore, with the help of Assumption \ref{A2.1}(a) and
(\ref{A2.1.2}), we have
\begin{align*}
\|e_{k+1}^\delta\|\le \|e_0\| +\frac{1}{2}K_0\|e_k^\delta\|^2 +c_4
\delta \alpha_k^{-1/2} \le (1+c_4 \gamma_1)
\|e_0\|+\frac{1}{2}K_0\|e_k^\delta\|^2.
\end{align*}
Thus, if $2(1+c_4 \gamma_1)K_0\|e_0\|\le 1$,  then, by using
$\rho>2(1+c_4 \gamma_1)\|e_0\|$ and an induction argument, we can
conclude $\|e_k^\delta\| \le 2(1+c_4 \gamma_1) \|e_0\|<\rho$ for
all $0\le k\le n_\delta$. This establishes (\ref{11.2}).

Next we show (\ref{11.2.5}). It follows from (\ref{2.7}),
Assumption \ref{A2.1}(a), (\ref{1.3}), (\ref{F2.2}) and
(\ref{11.1}) that for $0\le k<n_\delta$
\begin{align*}
\|F'(x_k^\delta) e_{k+1}^\delta\|&\lesssim \alpha_k^{1/2}
\|e_0\|+\delta
+\|F(x_k^\delta)-y-F'(x_k^\delta)e_k^\delta\|\\
&\lesssim  \alpha_k^{1/2}\|e_0\|+ (K_1+K_2)\|e_k^\delta\|
\|F'(x^\dag) e_k^\delta\|.
\end{align*}
By Assumption \ref{F2} we have
\begin{align*}
\|[F'(x^\dag)-F'(x_k^\delta)]e_{k+1}^\delta\| &\le
K_1\|e_k^\delta\|\|F'(x^\dag)e_{k+1}^\delta\|
+K_2\|e_{k+1}^\delta\|\|F'(x^\dag)e_k^\delta\|.
\end{align*}
The above two inequalities and (\ref{11.2}) then imply
$$
\|F'(x^\dag) e_{k+1}^\delta\|\lesssim
\alpha_k^{1/2}\|e_0\|+K_1\|e_0\|\|F'(x^\dag) e_{k+1}^\delta\|
+(K_1+K_2)\|e_0\| \|F'(x^\dag) e_k^\delta\|.
$$
Thus, if $(K_1+K_2)\|e_0\|$ is sufficiently small, we can conclude
(\ref{11.2.5}) by an induction argument. As direct consequences of
(\ref{11.2}), (\ref{11.2.5}) and Assumption \ref{F2} we have
\begin{equation}\label{4.28.1}
\|F'(x_k^\delta) e_k^\delta\|\lesssim \alpha_k^{1/2} \|e_0\|,
\quad 0\le k\le n_\delta
\end{equation}
and
\begin{equation}\label{4.28.1.1}
\|F'(x_{k+1}^\delta)(x_{k+1}^\delta-x_k^\delta)\|\lesssim
\alpha_k^{1/2} \|e_0\|, \quad 0\le k<n_\delta.
\end{equation}

In order to show $k_\delta\le n_\delta$, we note that (\ref{2.7})
gives
\begin{align*}
F'(x^\dag)e_{k+1}^\delta&-y^\delta+y = F'(x_k^\delta)
r_{\alpha_k}(\A_k^\delta) e_0 +\left(F'(x^\dag)-F'(x_k^\delta)\right)
r_{\alpha_k}(\A_k^\delta) e_0\\
&-\left(F'(x^\dag)-F'(x_k^\delta)\right)
g_{\alpha_k}(\A_k^\delta)F'(x_k^\delta)^*
\left(F(x_k^\delta)-y^\delta
-F'(x_k^\delta)e_k^\delta\right)\\
&-g_{\alpha_k}(\B_k^\delta)\B_k^\delta
\left(F(x_k^\delta)-y-F'(x_k^\delta) e_k^\delta\right)
-r_{\alpha_k}(\B_k^\delta)(y^\delta-y).
\end{align*}
Thus, by using (\ref{1.3}), Assumption \ref{A2.1}(a),
(\ref{A2.1.2}), Assumption \ref{F2}, (\ref{F2.1}), (\ref{11.2}),
(\ref{4.28.1}) and (\ref{4}) we have for $0\le k<n_\delta$
\begin{align*}
\|F'(x^\dag) e_{k+1}^\delta-y^\delta+y\|  &\le
\delta+c_3\alpha_k^{1/2}\|e_0\|+ c_3
K_1\|e_0\|\|e_k^\delta\|\alpha_k^{1/2}
+K_2\|e_0\|\|F'(x_k^\delta)e_k^\delta\|\\
&\quad + K_1\|e_k^\delta\|\left(\delta+\frac{1}{2} (K_1+K_2)
\|e_k^\delta\|\|F'(x_k^\delta) e_k^\delta\|\right)\\
&\quad + c_4 K_2 \alpha_k^{-1/2} \|F'(x_k^\delta) e_k^\delta\|
\left(\delta+\frac{1}{2}(K_1+K_2) \|e_k^\delta\|\|F'(x_k^\delta)
e_k^\delta\|\right)\\
&\quad +\frac{1}{2} (K_1+K_2)\|e_k^\delta\|\|F'(x_k^\delta) e_k^\delta\|\\
& \le \delta+\left(c_3+C(K_1+K_2) \|e_0\|\right) \alpha_k^{1/2}\|e_0\|\\
& \le \delta +r^{1/2} \left( c_3+C(K_1+K_2)\|e_0\|\right)
\alpha_{k+1}^{1/2} \|e_0\|.
\end{align*}
Recall that $\gamma_1>c_3 r^{1/2}/(\tau-1)$. Thus, with the help
of (\ref{11.2}), (\ref{11.2.5}) and the definition of $n_\delta$,
one can see that, if $(K_1+K_2)\|e_0\|$ is sufficiently small,
then
\begin{align*}
\|F(x_{n_\delta}^\delta)-y^\delta\| &\le
\|F(x_{n_\delta}^\delta)-y-F'(x^\dag) e_{n_\delta}^\delta\|
+\|F'(x^\dag) e_{n_\delta}^\delta -y^\delta+y\|\\
&\le \delta+ r^{1/2}\left(c_3+C(K_1+K_2)\|e_0\|\right)
\alpha_{n_\delta}^{1/2}\|e_0\|\\
&\quad \,\,
+\frac{1}{2}(K_1+K_2)\|e_{n_\delta}^\delta\|\|F'(x^\dag)
e_{n_\delta}^\delta\|\\
&\le \delta+
r^{1/2}\left(c_3+C(K_1+K_2)\|e_0\|\right)\alpha_{n_\delta}^{1/2}
\|e_0\|\\
&\le \delta+ r^{1/2}\left(c_3+C(K_1+K_2)\|e_0\|\right)
\gamma_1^{-1}\delta\\
&\le \tau\delta.
\end{align*}
This implies $k_\delta\le n_\delta$.

{\it Step 2}. We will show, for the noise-free iterated solutions
$\{x_k\}$, that for all $k\ge 0$
\begin{equation}\label{12.4}
\|r_{\alpha_k}(\A)e_0\|\lesssim  \|e_k\|\lesssim
\|r_{\alpha_k}(\A)e_0\|,
\end{equation}
\begin{equation}\label{12.5}
\|e_k\|\lesssim  \|e_{k+1}\|\lesssim \|e_k\|
\end{equation}
and for all $0\le k\le l$
\begin{equation}\label{12.6}
\|e_k\|\lesssim \|e_l\|+\frac{1}{\sqrt{\alpha_l}} \|F(x_k)-y\|.
\end{equation}

In fact, from (\ref{3.9}) and Assumption \ref{F1} it is easy to
see that
\begin{equation}\label{12.1}
\|e_{k+1}-r_{\alpha_k}(\A_k) e_0\|\le \frac{1}{2}K_0\|e_k\|^2.
\end{equation}
If $2 K_0\|e_0\|\le 1$, then by induction we can see that
$\{x_k\}$ is well-defined and
\begin{equation}\label{12.2}
\|e_k\|\le 2\|e_0\| \qquad \mbox{for all } k\ge 0.
\end{equation}
This together with  (\ref{12.1}) and (\ref{6.2.1}) gives
\begin{align}\label{12.3}
\|e_{k+1}-r_{\alpha_k}(\A) e_0\| \lesssim
\|[r_{\alpha_k}(\A_k)-r_{\alpha_k}(\A)]e_0\| +K_0\|e_k\|^2
\lesssim K_0\|e_0\|\|e_k\|.
\end{align}
Thus, by Assumption \ref{A2.2} and the smallness of $K_0\|e_0\|$
we obtain (\ref{12.4}) by induction. (\ref{12.5}) is an immediate
consequence of (\ref{12.3}) and (\ref{12.4}).

In order to show (\ref{12.6}), we first consider the case $k>0$.
Note that $x_k-x_l$ has a similar expression as in (\ref{12.10}),
so we may use (\ref{6.2.1}), Assumption \ref{F1} and (\ref{12.2})
to obtain
\begin{align}\label{12.11}
\|x_k-x_l\|& \lesssim
\|r_{\alpha_{k-1}}(\A)e_0-r_{\alpha_{l-1}}(\A)e_0\|
+K_0\|e_0\|\left(\|e_{k-1}\|+\|e_{l-1}\|\right)\nonumber\\
&\quad + K_0\|e_{k-1}\|^2+ K_0\|e_{l-1}\|^2\nonumber\\
& \lesssim \|[r_{\alpha_{k-1}}(\A)-r_{\alpha_{l-1}}(\A)]e_0\|
+K_0\|e_0\|\left(\|e_{k-1}\|+\|e_{l-1}\|\right).
\end{align}
By Lemma \ref{L10.1} with $x=e_0$, $\bar{x}=e_k$,
$\alpha=\alpha_{l-1}$, $\beta=\alpha_{k-1}$ and $A=F'(x^\dag)$, we
have
\begin{align*}
\|[r_{\alpha_{k-1}}(\A)-r_{\alpha_{l-1}}(\A)]e_0\| & \lesssim
\|r_{\alpha_{k-1}}(\A)e_0-e_k\|
+\frac{1}{\sqrt{\alpha_{l-1}}}\|F'(x^\dag)e_k\|.
\end{align*}
With the help of (\ref{F2.1}), (\ref{12.2}), and the smallness of
$(K_1+K_2)\|e_0\|$, we have
\begin{equation}\label{4.28.5}
\|F'(x^\dag)e_k\| \le \|F(x_k)-y\|+\frac{1}{2}\|F'(x^\dag)e_k\|.
\end{equation}
Therefore $\|F'(x^\dag)e_k\|\le 2\|F(x_k)-y\|$. This together with
(\ref{12.3}) and (\ref{12.5}) then implies
$$
\|[r_{\alpha_{k-1}}(\A)-r_{\alpha_{l-1}}(\A)]e_0\|\lesssim
K_0\|e_0\| \|e_k\| +\frac{1}{\sqrt{\alpha_l}}\|F(x_k)-y\|.
$$
Combining this with (\ref{12.11}) gives
\begin{align*}
\|x_k-x_l\| &\lesssim  K_0\|e_0\|\|e_k\|
+\|e_l\|+\frac{1}{\sqrt{\alpha_l}}\|F(x_k)-y\|
\end{align*}
which implies (\ref{12.6}) if $K_0\|e_0\|$ is sufficiently small.

For the case $k=0$, we can assume $l\ge 1$. Since (\ref{12.6}) is
valid for $k=1$, we may use (\ref{12.5}) to conclude that
(\ref{12.6}) is also true for $k=0$.

{\it Step 3}. We will show for all $k\ge 0$ that
\begin{equation}\label{5.5.0}
\|F'(x^\dag) e_k\| \lesssim \|r_{\alpha_k}(\A) \A^{1/2} e_0\|
+\alpha_k^{1/2}\|r_{\alpha_k}(\A)e_0\|.
\end{equation}

To this end, first we may use the similar manner in deriving
(\ref{11.2.5}) to conclude
\begin{equation}\label{5.5.1}
\|F'(x^\dag) e_k\|\lesssim \alpha_k^{1/2} \|e_0\|.
\end{equation}
Note that Assumption \ref{F2} and (\ref{12.2}) imply
\begin{align*}
\|[F'(x^\dag)-F'(x_k)]e_k\|
& \le (K_1+K_2)\|e_k\|\|F'(x^\dag) e_k\| \\
& \lesssim (K_1+K_2)\|e_0\| \|F'(x^\dag) e_k\|.
\end{align*}
Therefore
\begin{equation}\label{6.20.2}
\|F'(x_k)e_k\|\lesssim \|F'(x^\dag) e_k\|.
\end{equation}
In particular this implies
\begin{equation}\label{5.5.2}
\|F'(x_k) e_k\|\lesssim  \alpha_k^{1/2} \|e_0\|.
\end{equation}
By using (\ref{4.27.3}), (\ref{6.2.2}), Assumption \ref{F2},
(\ref{F2.1}) and Assumption \ref{A2.1}(a) we obtain
\begin{align*}
\|F'(x^\dag) e_{k+1}\| &\lesssim  \|r_{\alpha_k}(\A)\A^{1/2}
e_0\|+
(K_0+K_1)\|e_0\|\|e_k\| \alpha_k^{1/2}\\
&+K_2\|e_0\|\left(\|F'(x^\dag)e_k\|+ \|F'(x_k)e_k\|\right)\\
& +(K_1+K_2)\|e_k\|\|F'(x_k) e_k\|
+ K_1(K_1+K_2) \|e_k\|^2\|F'(x_k) e_k\|\\
&+ K_2(K_1+K_2)\|e_k\|\|F'(x_k)e_k\|^2\alpha_k^{-1/2}.
\end{align*}
Thus, with the help of (\ref{12.4}), (\ref{5.5.1}), (\ref{6.20.2}),
(\ref{5.5.2}) and (\ref{12.2}), we obtain
\begin{align*}
\|F'(x^\dag)e_{k+1}\|\lesssim  \|r_{\alpha_k}(\A)\A^{1/2}
e_0\|+\alpha_k^{1/2} \|r_{\alpha_k}(\A)e_0\|+ K_2\|e_0\|
\|F'(x^\dag) e_k\|.
\end{align*}
The estimates (\ref{5.5.0}) thus follows by Assumption \ref{A2.2}
and an induction argument if $K_2\|e_0\|$ is sufficiently small.

{\it Step 4}. Now we will establish some stability estimates. We
will show for all $0\le k\le n_\delta$ that
\begin{equation}\label{4.27.5}
\|x_k^\delta-x_k\|\lesssim \frac{\delta}{\sqrt{\alpha_k}}
\end{equation}
and
\begin{equation}\label{4.27.6}
\|F(x_k^\delta)-F(x_k)-y^\delta+y\|\le \left(1+
C(K_0+K_1+K_2)\|e_0\|\right)\delta.
\end{equation}

In order to show (\ref{4.27.5}), we use again the decomposition
(\ref{3.53}) for $x_{k+1}^\delta-x_{k+1}$. We still have
$\|I_2\|\le c_4\delta/\sqrt{\alpha_k}$. By using (\ref{6.2.1}) the
term $I_1$ can be estimated as
$$
\|I_1\|\lesssim  K_0\|e_0\|\|x_k^\delta-x_k\|.
$$
In order to estimate $I_3$, we note that Assumption \ref{F1}
implies
\begin{align*}
I_3 & = \int_0^1 \left[ g_{\alpha_k}(\A_k)\A_k
-g_{\alpha_k}(\A_k^\delta)\A_k^\delta\right]
\left[R(x_k-te_k, x_k)-I\right] e_k dt\\
& + \int_0^1 g_{\alpha_k}(\A_k^\delta)
F'(x_k^\delta)^*\left[F'(x_k^\delta)-F'(x_k)\right]
\left[R(x_k-te_k, x_k)-I\right] e_k dt\\
& =\int_0^1
\left[r_{\alpha_k}(\A_k^\delta)-r_{\alpha_k}(\A_k)\right]
\left[R(x_k-te_k, x_k)-I\right] e_k dt\\
& +\int_0^1 g_{\alpha_k}(\A_k^\delta) \A_k^\delta\left[I-R(x_k,
x_k^\delta)\right] \left[R(x_k-te_k, x_k)-I\right] e_k dt.
\end{align*}
Thus, by using (\ref{6.2.1}) and (\ref{12.2}), we obtain
$$
\|I_3\|\lesssim  K_0^2\|e_k\|^2\|x_k^\delta-x_k\| \lesssim
K_0^2\|e_0\|^2\|x_k^\delta-x_k\|.
$$
In order to estimate $I_4$, we again use Assumption \ref{F1} to
write
\begin{align*}
I_4 & = g_{\alpha_k}(\A_k^\delta)F'(x_k^\delta)^*
\left[F(x_k)-F(x_k^\delta)-F'(x_k^\delta)(x_k-x_k^\delta)\right]\\
&+g_{\alpha_k}(\A_k^\delta)F'(x_k^\delta)^*
\left[F'(x_k^\delta)-F'(x_k)\right] e_k\\
& = \int_0^1 g_{\alpha_k}(\A_k^\delta)\A_k^\delta
\left[R(x_k^\delta+t(x_k-x_k^\delta),
x_k^\delta)-I\right](x_k-x_k^\delta) dt\\
& + g_{\alpha_k}(\A_k^\delta)\A_k^\delta\left[I-R(x_k,
x_k^\delta)\right] e_k.
\end{align*}
Hence, we may use (\ref{11.2}) and (\ref{12.2}) to derive that
\begin{align*}
\|I_4\|&\lesssim
K_0\|x_k^\delta-x_k\|^2+K_0\|e_k\|\|x_k^\delta-x_k\| \lesssim
K_0\|e_0\|\|x_k^\delta-x_k\|.
\end{align*}
Combining the above estimates we obtain for $0\le k< n_\delta$
$$
\|x_{k+1}^\delta-x_{k+1}\|\lesssim
\frac{\delta}{\sqrt{\alpha_k}} + K_0\|e_0\| \|x_k^\delta-x_k\|.
$$
Thus, if $K_0\|e_0\|$ is sufficiently small, we can obtain
(\ref{4.27.5}) immediately.

Next we show (\ref{4.27.6}) by using (\ref{4.27.8}). We still have
(\ref{4.27.9}). In order to estimate $\|F'(x_k^\delta) I_1\|$,
$\|F'(x_k^\delta)I_3\|$ and $\|F'(x_k^\delta) I_4\|$, we note that
Assumption \ref{F2}, (\ref{12.2}), (\ref{5.5.1}) and
(\ref{4.27.5}) imply
\begin{align*}
\|[F'(x_k)&-F'(x^\dag)](x_k^\delta-x_k)\|\\
&\le  K_1\|e_k\|\|F'(x^\dag)(x_k^\delta-x_k)\|+K_2\|F'(x^\dag)
e_k\|\|x_k^\delta-x_k\|\\
&\lesssim  K_1\|e_0\| \|F'(x^\dag)(x_k^\delta-x_k)\|+
K_2\|e_0\|\delta,
\end{align*}
which in turn gives
\begin{equation}\label{4.28.6}
\|F'(x_k)(x_k^\delta-x_k)\|\lesssim  \|F'(x^\dag)
(x_k^\delta-x_k)\|+\delta.
\end{equation}
Similarly, we have
\begin{equation}\label{4.28.6.1}
\|F'(x_k^\delta)(x_k^\delta-x_k)\|\lesssim  \|F'(x^\dag)
(x_k^\delta-x_k)\|+\delta.
\end{equation}
Thus, by using (\ref{6.2.2}), (\ref{4.27.5}), (\ref{4.28.6})
and (\ref{4.28.6.1}) we have
\begin{align*}
\|F'(x_k^\delta) I_1\| &\lesssim  (K_0+ K_1)\|e_0\|\alpha_k^{1/2}
\|x_k^\delta-x_k\|\\
&\quad  +K_2\|e_0\|\left(\|F'(x_k^\delta)(x_k^\delta-x_k)\|
+\|F'(x_k)(x_k^\delta-x_k)\|\right)\\
& \lesssim (K_0+K_1+K_2)\|e_0\|\delta+ K_2\|e_0\|
\|F'(x^\dag)(x_k^\delta-x_k)\|.
\end{align*}
Moreover, by employing (\ref{4.28.3}), (\ref{6.2.1}), Assumption
\ref{F2}, (\ref{F2.1}), (\ref{12.2}), (\ref{5.5.2}),
(\ref{4.27.5}) and (\ref{4.28.6}), $\|F'(x_k^\delta) I_3\|$ can be
estimated as
\begin{align*}
\|F'(x_k^\delta) I_3\|& \lesssim (K_0+K_1)\|x_k^\delta-x_k\|\|u_k\|
 +  \alpha_k^{-1/2} K_2\|F'(x_k)(x_k^\delta-x_k)\|\|u_k\|\\
& \lesssim  (K_0+ K_1+K_2)(K_1+K_2) \|e_0\|^2 \delta\\
&\quad + K_2(K_1+K_2) \|e_0\|^2 \|F'(x^\dag)(x_k^\delta-x_k)\|.
\end{align*}
while, by using Assumption \ref{F2}, (\ref{F2.1}), (\ref{11.2}),
(\ref{12.2}), (\ref{4.28.1}), (\ref{4.27.5}), (\ref{4.28.6}) and
(\ref{4.28.6.1}), $\|F'(x_k^\delta) I_4\|$ can be estimated as
\begin{align*}
\|F'(x_k^\delta) I_4\| & \le
\|F(x_k^\delta)-F(x_k)-F'(x_k)(x_k^\delta-x_k)\|
+\|[F'(x_k^\delta)-F'(x_k)]e_k^\delta\|\\
& \lesssim (K_1+K_2)\|x_k^\delta-x_k\|
\|F'(x_k)(x_k^\delta-x_k)\|\\
&\quad +K_1\|x_k^\delta-x_k\|\|F'(x_k^\delta) e_k^\delta\|
+K_2\|F'(x_k^\delta)(x_k^\delta-x_k)\|\|e_k^\delta\|\\
& \lesssim (K_1+K_2) \|e_0\| \delta + (K_1+K_2)\|e_0\|
\|F'(x^\dag)(x_k^\delta-x_k)\|.
\end{align*}
Combining the above estimates we get
\begin{align}\label{4.28.7}
\|F'(&x_k^\delta)(x_{k+1}^\delta-x_{k+1})-y^\delta+y\| \nonumber\\
&\le  (1+C(K_0+K_1+K_2)\|e_0\|)  \delta + C (K_1+K_2)\|e_0\|
\|F'(x^\dag)(x_k^\delta-x_k)\|.
\end{align}
This in particular implies
$$
\|F'(x_k^\delta)(x_{k+1}^\delta-x_{k+1})\|\lesssim
\delta+(K_1+K_2)\|e_0\|\|F'(x^\dag)(x_k^\delta-x_k)\|.
$$
On the other hand, similar to the derivation of (\ref{4.28.6}), by
Assumption \ref{F2}, (\ref{11.2}), (\ref{4.28.1}) and (\ref{4.27.5})
we have for $0\le k<n_\delta$ that
$$
\|F'(x^\dag)(x_{k+1}^\delta-x_{k+1})\|\lesssim K_2\|e_0\|\delta
+\|F'(x_k^\delta)(x_{k+1}^\delta-x_{k+1})\|.
$$
Therefore
$$
\|F'(x^\dag)(x_{k+1}^\delta-x_{k+1})\|\lesssim
\delta+(K_1+K_2)\|e_0\|\|F'(x^\dag)(x_k^\delta-x_k)\|.
$$
Thus, if $(K_1+K_2)\|e_0\|$ is small enough, then we can conclude
\begin{equation}\label{6.27.1}
\|F'(x^\dag)(x_k^\delta-x_k)\|\lesssim \delta, \qquad 0\le k\le
n_\delta.
\end{equation}
Combining this with (\ref{4.28.7}) gives for $0\le k<n_\delta$
\begin{equation}\label{6.16.1}
\|F'(x_k^\delta)(x_{k+1}^\delta-x_{k+1})-y^\delta+y\|\le
\left(1+C(K_0+K_1+K_2)\|e_0\|\right) \delta.
\end{equation}
Hence, by using (\ref{6.16.1}), Assumption \ref{F2}, (\ref{11.2}),
(\ref{4.28.1.1}), (\ref{4.27.5}), (\ref{4.28.6.1}) and (\ref{6.27.1}),
we obtain for $0\le k\le n_\delta$
$$
\|F'(x_k^\delta)(x_k^\delta-x_k)-y^\delta+y\|\le
\left(1+C(K_0+K_1+K_2)\|e_0\|\right) \delta.
$$
This together with (\ref{F2.1}), (\ref{11.2}) and (\ref{12.2})
implies (\ref{4.27.6}).

{\it Step 5}. Now we are ready to complete the proof. By using the
definition of $k_\delta$, (\ref{4.27.6}), (\ref{F2.1}) and
(\ref{5.5.0}) we have for $0\le k<k_\delta$
\begin{align*}
\tau \delta &\le \|F(x_k^\delta)-y^\delta\|\le
\|F(x_k^\delta)-F(x_k)-y^\delta+y\|+\|F(x_k)-y\|\\
&\le (1+C(K_0+K_1+K_2)\|e_0\|)\delta+ C \|F'(x^\dag)e_k\|\\
&\le \left(1+C(K_0+K_1+K_2)\|e_0\|\right)\delta +
C\|r_{\alpha_k}(\A) \A^{1/2} e_0\|
+C\alpha_k^{1/2}\|r_{\alpha_k}(\A)e_0\|.
\end{align*}
Since $\tau>1$, by assuming $(K_0+K_1+K_2)\|e_0\|$ is small
enough, we can conclude for $0\le k<k_\delta$ that
\begin{equation}\label{6.17.3}
(\tau-1)\delta\lesssim \|r_{\alpha_k}(\A) \A^{1/2} e_0\|
+\alpha_k^{1/2}\|r_{\alpha_k}(\A)e_0\|.
\end{equation}

When $x_0-x^\dag$ satisfies (\ref{7}) for some $\omega\in X$ and
$0<\nu\le \bar{\nu}-1/2$, by using (\ref{6.17.3}), (\ref{12.6}),
(\ref{12.4}), (\ref{4.27.5}), (\ref{4.27.6}) and the definition of
$k_\delta$, we can employ the similar argument as in the last part
of the proof of Theorem \ref{T4.1} to conclude (\ref{T4.5.1}).

When $x_0-x^\dag$ satisfies (\ref{logsource}) for some $\omega\in
X$ and $\mu>0$, we have from Assumption \ref{A2.1}(a) and
(\ref{8.25.2}) that
$$
\|r_{\alpha_k}(\A) \A^{1/2} e_0\|
+\alpha_k^{1/2}\|r_{\alpha_k}(\A)e_0\|\le \left(c_0
b_{2\mu}^{1/2}+b_\mu\right) \alpha_k^{1/2}
\left(-\ln(\alpha_k/(2\alpha_0))\right)^{-\mu}\|\omega\|.
$$
This and (\ref{6.17.3}) imply that there exists a constant
$C_\mu>0$ such that
$$
(\tau-1)\delta< C_\mu \alpha_k^{1/2} \left(-\ln
(\alpha_k/(2\alpha_0))\right)^{-\mu}\|\omega\|, \quad 0\le
k<k_\delta.
$$
If we introduce the integer $\hat{k}_\delta$ by
$$
\alpha_{\hat{k}_\delta}^{1/2} \left(-\ln
(\alpha_{\hat{k}_\delta}/(2\alpha_0))\right)^{-\mu} \le
\frac{(\tau-1)\delta}{C_\mu\|\omega\|}<\alpha_k^{1/2} \left(-\ln
(\alpha_k/(2\alpha_0))\right)^{-\mu}, \quad 0\le k<\hat{k}_\delta,
$$
then $k_\delta\le \hat{k}_\delta$. Thus, by using (\ref{12.6}),
(\ref{4.27.5}), (\ref{4.27.6}), the definition of $k_\delta$ and
the fact $\|e_k\|\lesssim \|r_{\alpha_k}(\A) e_0\|\lesssim
(-\ln(\alpha_k/(2\alpha_0)))^{-\mu}\|\omega\|$, we can use the
similar manner in deriving (\ref{6.18.1}) to get
\begin{equation}\label{6.18.9}
\|e_{k_\delta}^\delta\|\lesssim
\left(-\ln(\alpha_{\hat{k}_\delta}/(2\alpha_0))\right)^{-\mu}
\|\omega\|+\frac{\delta}{\sqrt{\alpha_{\hat{k}_\delta}}} \lesssim
\frac{\delta}{\sqrt{\alpha_{\hat{k}_\delta}}}.
\end{equation}
By elementary argument we can show from (\ref{4}) and the
definition of $\hat{k}_\delta$ that there is a constant $c_\mu>0$
such that
$$
\alpha_{\hat{k}_\delta}\ge r^{-1} \alpha_{\hat{k}_\delta-1}\ge
c_\mu \left(\frac{\delta}{\|\omega\|}\right)^2 \left(1+\left|\ln
\frac{\delta}{\|\omega\|}\right|\right)^{2\mu}.
$$
This together with (\ref{6.18.9}) implies the estimate
(\ref{T4.5.2}). \hfill $\Box$
\end{proof}

\section{\bf Proof of Theorem \ref{T8.1}}
\setcounter{equation}{0}

If $x_0=x^\dag$, then $k_\delta=0$ and the result is trivial.
Therefore, we will assume $x_0\ne x^\dag$. We define
$\hat{k}_\delta$ to be the first integer such that
$$
\|r_{\alpha_{\hat{k}_\delta}}(\A)\A^{1/2} e_0\|
+\alpha_{\hat{k}_\delta}^{1/2}\|r_{\alpha_{\hat{k}_\delta}}(\A)e_0\|
\le c \delta,
$$
where the constant $c>0$ is chosen so that we may apply Lemma
\ref{L4.2} or (\ref{6.17.3}) to conclude $k_\delta\le
\hat{k}_\delta$. By (\ref{4}),  such $\hat{k}_\delta$ is clearly
well-defined and is finite. Moreover, by a contradiction argument
it is easy to show that
\begin{equation}\label{8.1.5}
\hat{k}_\delta\rightarrow \infty \quad \mbox{as }
\delta\rightarrow 0.
\end{equation}

Now, under the conditions of Theorem \ref{T8.1} (i) we use  Lemma
\ref{L3.2}, Lemma \ref{L3.3} and (\ref{3.3}), while under the
conditions of Theorem \ref{T8.1} (ii) we use (\ref{4.27.5}),
(\ref{4.27.6}), (\ref{12.4}) and (\ref{12.6}), then from
the definition of $k_\delta$ we have
\begin{align}\label{8.6}
\|e_{k_\delta}^\delta\|&\lesssim
\|e_{k_\delta}\|+\frac{\delta}{\sqrt{\alpha_{k_\delta}}} \lesssim
\|e_{k_\delta}\|+\frac{\delta}{\sqrt{\alpha_{\hat{k}_\delta}}}
\nonumber\\
&\lesssim \|e_{\hat{k}_\delta}\|
+\frac{1}{\sqrt{\alpha_{\hat{k}_\delta}}}
\left(\|F(x_{k_\delta})-y\|+\delta\right)\nonumber\\
&\lesssim \|r_{\alpha_{\hat{k}_\delta}}(\A) e_0\|
+\frac{\delta}{\sqrt{\alpha_{\hat{k}_\delta}}}\nonumber\\
&\lesssim \frac{\delta}{\sqrt{\alpha_{\hat{k}_\delta}}}.
\end{align}
We therefore need to derive the lower bound of
$\alpha_{\hat{k}_\delta}$ under the conditions on $e_0$. We set
for each $\alpha>0$ and $0\le \mu\le \bar{\nu}$
$$
c_\mu(\alpha):=\left[\int_0^{1/2} \alpha^{-2\mu}
r_\alpha(\lambda)^2 \lambda^{2\mu} d(E_\lambda \omega,
\omega)\right]^{1/2},
$$
where $\{E_\lambda\}$ denotes the spectral family generated by
$\A$. It is easy to see for each $0\le \mu<\bar{\nu}$ that $
\alpha^{-2\mu} r_{\alpha}(\lambda)^2 \lambda^{2\mu}$ is uniformly
bounded for all $\alpha>0$ and $\lambda\in [0, 1/2]$ and
$\alpha^{-2\mu}r_{\alpha}(\lambda)^2\lambda^{2\mu}\rightarrow 0$
as $\alpha\rightarrow 0$ for all $\lambda\in (0, 1/2]$. Since
$\omega\in {\mathcal N}(F'(x^\dag))^\perp$, by the dominated
convergence theorem we have for each $0\le \mu<\bar{\nu}$
\begin{equation}\label{8.3}
c_\mu(\alpha)\rightarrow 0 \quad \mbox{as } \alpha\rightarrow 0.
\end{equation}
By the definition of $\hat{k}_\delta$, (\ref{4}), Assumption
\ref{A2.2}, and the condition $e_0=\A^\nu \omega$ we have
\begin{align*}
\delta &\lesssim \|r_{\alpha_{\hat{k}_\delta-1}}(\A)\A^{1/2}e_0\|
+\alpha_{\hat{k}_\delta-1}\|r_{\alpha_{\hat{k}_\delta-1}}(\A) e_0\|\\
&\lesssim \|r_{\alpha_{\hat{k}_\delta}}(\A)\A^{1/2}e_0\|
+\alpha_{\hat{k}_\delta}\|r_{\alpha_{\hat{k}_\delta}}(\A) e_0\|\\
&\lesssim \alpha_{\hat{k}_\delta}^{\nu+1/2}
\left(c_\nu(\alpha_{\hat{k}_\delta})+c_{\nu+1/2}(\alpha_{\hat{k}_\delta})\right)
\end{align*}
This implies
\begin{equation}\label{8.5}
\alpha_{\hat{k}_\delta}\ge
\left(\frac{c\delta}{c_\nu(\alpha_{\hat{k}_\delta})
+c_{\nu+1/2}(\alpha_{\hat{k}_\delta})}\right)^{2/(1+2\nu)}.
\end{equation}
Combining (\ref{8.6}) and (\ref{8.5}) gives
\begin{equation*}
\|e_{k_\delta}^\delta\|\lesssim
\left(c_\nu(\alpha_{\hat{k}_\delta})
+c_{\nu+1/2}(\alpha_{\hat{k}_\delta})\right)^{1/(1+2\nu)}
\delta^{2\nu/(1+2\nu)}
\end{equation*}
Since $0\le \nu<\bar{\nu}-1/2$, this together with (\ref{8.1.5})
and (\ref{8.3}) gives the desired conclusion.

\section{\bf Applications}
\setcounter{equation}{0}

In this section we will consider some specific methods defined by
(\ref{3}) by presenting several examples of $\{g_\alpha\}$. We
will verify that those assumptions in Section 2 are satisfied for
these examples.

\subsection{\bf Example 1} We first consider the function
$g_\alpha$ given by
\begin{equation}\label{g1}
g_\alpha(\lambda)
=\frac{(\alpha+\lambda)^m-\alpha^m}{\lambda(\alpha+\lambda)^m},
\end{equation}
where $m\ge 1$ is a fixed integer. This function arises from the
iterated Tikhonov regularization of order $m$ for linear ill-posed
problems. Note that when $m=1$, the corresponding method defined
by (\ref{3}) is exactly the iteratively regularized Gauss-Newton
method (\ref{IRGN}). It is clear that the residual function
corresponding to (\ref{g1}) is
$$
r_\alpha(\lambda)=\frac{\alpha^m}{(\alpha+\lambda)^m}.
$$
By elementary calculations it is easy to see that Assumption
\ref{A2.1}(a) and (b) are satisfied with $ c_0=(m-1)^{m-1}/m^m$
and $c_1=m$. Moreover (\ref{A2.1.2}) is satisfied with
$$
c_3=\frac{1}{\sqrt{2m-1}}\left(\frac{2m-1}{2m}\right)^m \quad
\mbox{and}\quad
c_4=\left(1-\left(\frac{m+1}{m+3}\right)^m\right)\sqrt{m}.
$$
By using the elementary inequality
\begin{equation}\label{5.10}
1-(1-t)^n\le \sqrt{nt}, \qquad 0\le t\le 1
\end{equation}
for any integer $n\ge 0$, we have for $0<\alpha\le \beta$ and
$\lambda\ge 0$ that
$$
r_\beta(\lambda)-r_\alpha(\lambda) =r_\beta(\lambda)
\left[1-\left(1-\frac{\lambda/\alpha
-\lambda/\beta}{1+\lambda/\alpha}\right)^m\right] \le  m^{1/2}
\sqrt{\frac{\lambda}{\alpha}} r_\beta(\lambda).
$$
This verifies Assumption \ref{A2.1}(c) with $c_2=m^{1/2}$. It is
well-known that the qualification for $g_\alpha$ is $\bar{\nu}=m$
and (\ref{A2.1.3}) is satisfied with $d_\nu=(\nu/m)^\nu
((m-\nu)/m)^{m-\nu}\le 1$ for each $0\le \nu\le m$. For the
sequence $\{\alpha_k\}$ satisfying (\ref{4}), Assumption
\ref{A2.2} is satisfied with $c_5=r^m$.

In order to verify Assumption \ref{A2.3}, we note that
\begin{align}\label{6.20.1}
r_{\alpha}&(A^*A)-r_\alpha(B^*B)\nonumber\\
&=\alpha^m \sum_{i=1}^m (\alpha
I+A^*A)^{-i}[A^*(B-A)+(B^*-A^*)B](\alpha I+B^*B)^{-m-1+i}.
\end{align}
Thus, by using the estimates
$$
\|(\alpha I+A^*A)^{-i}(A^*A)^{\mu}\|\le \alpha^{-i+\mu} \quad
\mbox{for } i\ge 1 \mbox{ and } 0\le \mu\le 1,
$$
we can verify (\ref{3.41}), (\ref{3.40}) and (\ref{3.42}) easily.

Note also that $g_\alpha(\lambda)=\alpha^{-1}\sum_{i=1}^m
\alpha^i(\alpha+\lambda)^{-i}$. We have, by using (\ref{3.40}),
\begin{align*}
\|[g_\alpha(A^*A)-g_\alpha(B^*B)]B^*\| &\le
\alpha^{-1}\sum_{i=1}^m \|\alpha^i [(\alpha I+A^*A)^{-i}-(\alpha
I+B^*B)^{-i}]B^*\|\\
& \lesssim   \alpha^{-1} \|A-B\|,
\end{align*}
which verifies (\ref{3.43}).

Finally we verify Assumption \ref{A6.2} by assuming that $F$
satisfies Assumption \ref{F1} and Assumption \ref{F2}. We will use
the abbreviation $F_x':=F'(x)$ for $x\in B_\rho(x^\dag)$. With the
help of (\ref{6.20.1}) with $A=F_x'$ and $B=F_z'$, we obtain from
Assumption \ref{F1} that
\begin{align*}
\|r_{\alpha}&(F_x'^*F_x')-r_{\alpha}(F_z'^*F_z')\|\\
&\le \alpha^m \sum_{i=1}^m \|(\alpha I+F_x'^*F_x')^{-i}
F_x'^*F_x'[R(z, x)-I](\alpha I+F_z'^*F_z')^{-m-1+i}\|\\
& +\alpha^m \sum_{i=1}^m \|(\alpha I+F_x'^*F_x')^{-i}
[I-R(x, z)]^*F_z'^*F_z'(\alpha I+F_z'^*F_z')^{-m-1+i}\|\\
&\le \alpha^m \sum_{i=1}^m \alpha^{-i+1}\|I-R(z,
x)\|\alpha^{-m-1+i}+\alpha^m \sum_{i=1}^m \alpha^{-i}\|I-R(x, z)\|
\alpha^{-m+i} \\
&\lesssim \|I-R(z, x)\|+\|I-R(x, z)\|\\
&\lesssim K_0\|x-z\|
\end{align*}
which verifies (\ref{6.2.1}). In order to show (\ref{6.2.2}), we
note that, for any $a\in X$ and $b\in Y$ satisfying
$\|a\|=\|b\|=1$, (\ref{6.20.1}) implies
\begin{align*}
(F_x'[r_\alpha&(F_x'^*F_x')-r_\alpha(F_z'^*F_z')]a, b)\\
& \le  \alpha^m \sum_{i=1}^m \alpha^{-i+1}\|(F_z'-F_x')(\alpha
I+F_z'^*F_z')^{-m-1+i}a\|\|b\|\\
&+\alpha^m \sum_{i=1}^m \alpha^{-m-1/2+i} \|(F_z'-F_x') (\alpha I+
F_x'^*F_x')^{-i}F_x'^* b\|\|a\|.
\end{align*}
Thus, by using Assumption \ref{F2}, we have
\begin{align*}
(F_x'[r_\alpha & (F_x'^*F_x')-r_\alpha(F_z'^*F_z')]a, b)\\
& \le  \alpha^m \sum_{i=1}^m \alpha^{-i+1} K_1\|x-z\|\|F_z'(\alpha
I +F_z'^*F_z')^{-m-1+i}a\|\\
&+\alpha^m \sum_{i=1}^m \alpha^{-i+1} K_2\|F_z'(x-z)\|\|(\alpha I
+F_z'^*F_z')^{-m-1+i}a\|\\
&+\alpha^m \sum_{i=1}^m \alpha^{-m-1/2+i} K_1\|x-z\|\|F_x'(\alpha
I +F_x'^*F_x')^{-i} F_x'^* b\|\\
&+\alpha^m \sum_{i=1}^m \alpha^{-m-1/2+i} K_2 \|F_x'(x-z)\|
\|(\alpha I+ F_x'^*F_x')^{-i} F_x'^* b\|\\
& \lesssim   K_1\alpha^{1/2} \|x-z\| +K_2\left(\|F_x'(x-z)\|
+\|F_z'(x-z)\|\right).
\end{align*}
This verifies (\ref{6.2.2}).

The above analysis shows that Theorem \ref{T4.1}, Theorem
\ref{T4.5} and Theorem \ref{T8.1} are applicable for the method
defined by (\ref{3}) and (\ref{2.4}) with $g_\alpha$ given by
(\ref{g1}). Thus we obtain the following result.

\begin{corollary}\label{C6.1}
Let $F$ satisfy (\ref{2.1}) and (\ref{2.3}), let $\{\alpha_k\}$ be
a sequence of numbers satisfying (\ref{4}), and let
$\{x_k^\delta\}$ be defined by (\ref{3}) with $g_\alpha$ given by
(\ref{g1}) for some fixed integer $m\ge 1$. Let $k_\delta$ be the
first integer satisfying (\ref{2.4}) with $\tau>1$.

(i) If $F$ satisfies Assumption \ref{Lip} and if $x_0-x^\dag$
satisfies (\ref{7}) for some $\omega\in X$ and $1/2\le \nu\le
m-1/2$, then
$$
\|x_{k_\delta}^\delta-x^\dag\| \le
C_\nu\|\omega\|^{1/(1+2\nu)}\delta^{2\nu/(1+2\nu)}
$$
provided $L\|u\|\le \eta_0$, where $u\in {\mathcal
N}(F'(x^\dag)^*)^\perp\subset Y$ is the unique element such that
$x_0-x^\dag=F'(x^\dag)^* u$, $\eta_0>0$ is a constant depending
only on $r$, $\tau$ and $m$,  and $C_\nu>0$ is a constant
depending only on $r$, $\tau$, $m$ and $\nu$.

(ii) Let $F$ satisfy Assumption \ref{F1} and Assumption \ref{F2},
and let $x_0-x^\dag \in N(F'(x^\dag))^\perp$. Then there exists a
constant $\eta_1>0$ depending only on $r$, $\tau$ and $m$ such
that if $(K_0+K_1+K_2)\|x_0-x^\dag\|\le \eta_1$ then
$$
\lim_{\delta\rightarrow 0} x_{k_\delta}^\delta=x^\dag,
$$
moreover, when $x_0-x^\dag$ satisfies (\ref{7}) for some
$\omega\in X$ and $0<\nu\le m-1/2$, then
$$
\|x_{k_\delta}^\delta-x^\dag\|\le C_\nu \|\omega\|^{1/(1+2\nu)}
\delta^{2\nu/(1+2\nu)}
$$
for some constant $C_\nu>0$ depending only on $r$, $\tau$, $m$ and
$\nu$; while when $x_0-x^\dag$ satisfies (\ref{logsource}) for
some $\omega\in X$ and $\mu>0$, then
\begin{equation*}
\|x_{k_\delta}^\delta-x^\dag\|\le C_\mu
\|\omega\|\left(1+\left|\ln
\frac{\delta}{\|\omega\|}\right|\right)^{-\mu}
\end{equation*}
for some constant $C_\mu$ depending only on $r$, $\tau$, $m$ and
$\mu$.
\end{corollary}

Corollary \ref{C6.1} with $m=1$ reproduces those convergence
results in \cite{BNS97,H97} for the iteratively regularized
Gauss-Newton method (\ref{IRGN}) together with the discrepancy
principle (\ref{2.4}) under somewhat different conditions on $F$.
Note that those results in \cite{BNS97,H97} require $\tau$ be
sufficiently large, while our result is valid for any $\tau>1$.
This less restrictive requirement on $\tau$ is important in
numerical computations since the absolute error could increase
with respect to $\tau$. Moreover, when $x_0-x^\dag$ satisfies
(\ref{7}) with $\nu=1/2$, Corollary \ref{C6.1} with $m=1$ improves
the corresponding result in \cite{BNS97}, since we only need the
Lipschitz condition on $F'$ here.

Corollary \ref{C6.1} shows that the method defined by (\ref{3})
and (\ref{2.4}) with $g_\alpha$ given by (\ref{g1}) is order
optimal for $0<\nu\le m-1/2$. However, we can not expect better
rate of convergence than $O(\delta^{(2m-1)/(2m)})$ even if
$x_0-x^\dag$ satisfies (\ref{7}) with $m-1/2<\nu\le m$. An a
posteriori stopping rule without such saturation has been studied
in \cite{Jin00,Jin08} for the iteratively regularized
Gauss-Newton method (\ref{IRGN}).

\subsection{\bf Example 2} We consider the function $g_\alpha$ given by
\begin{equation}\label{g2}
g_\alpha(\lambda)=\sum_{i=0}^{[1/\alpha]}(1-\lambda)^i
\end{equation}
which arises from the Landweber iteration applying to linear
ill-posed problems. With such choice of $g_\alpha$, the method
(\ref{3}) becomes
$$
x_{k+1}^\delta=x_0-\sum_{i=0}^{[1/\alpha_k]}\left(I-F'(x_k^\delta)^*
F'(x_k^\delta)\right)^i F'(x_k^\delta)^*
\left(F(x_k^\delta)-y^\delta-F'(x_k^\delta)(x_k^\delta-x_0)\right)
$$
which is equivalent to the form
\begin{align*}
&x_{k, 0}^\delta=x_0,\\
&x_{k, i+1}^\delta=x_{k, i}^\delta-F'(x_k^\delta)^*
\left(F(x_k^\delta)-y^\delta+F'(x_k^\delta) (x_{k,
i}^\delta-x_k^\delta)\right), \qquad 0\le i\le [1/\alpha_k],\\
&x_{k+1}^\delta=x_{k, [1/\alpha_k]+1}^\delta.
\end{align*}
This method has been considered in \cite{K97} and is called the
Newton-Landweber iteration.

Note that the corresponding residual function is
\begin{equation}\label{8.24.2008}
r_\alpha(\lambda)=(1-\lambda)^{[1/\alpha]+1}.
\end{equation}
It is easy to see that Assumption \ref{A2.1}(a), (b) and
(\ref{A2.1.2}) hold with
$$
c_0=\frac{1}{2}, \quad c_1=2, \quad c_3=\frac{\sqrt{2}}{3} \quad
\mbox{and} \quad c_4=\sqrt{2}.
$$
Moreover, by (\ref{5.10}) we have for any $0<\alpha\le \beta$ that
$$
r_\beta(\lambda)-r_\alpha(\lambda) =r_\beta(\lambda)
\left(1-(1-\lambda)^{[1/\alpha]-[1/\beta]}\right) \le
\sqrt{\frac{\lambda}{\alpha}}r_\beta(\lambda).
$$
This verifies Assumption \ref{A2.1}(c) with $c_2=1$. It is
well-known that the qualification of linear Landweber iteration is
$\bar{\nu}=\infty$ and (\ref{A2.1.3}) is satisfied with
$d_\nu=\nu^\nu$ for each $0\le \nu< \infty$.

In order to verify Assumption \ref{A2.2}, we restrict the sequence
$\{\alpha_k\}$ to be of the form $\alpha_k:=1/n_k$, where
$\{n_k\}$ is a sequence of positive integers such that
\begin{equation}\label{5.1}
0\le n_{k+1}-n_k\le q \quad \mbox{and}\quad \lim_{k\rightarrow
\infty} n_k=\infty
\end{equation}
for some $q\ge 1$. Then for $\lambda\in [0, 1/2]$ we have
$$
r_{\alpha_k}(\lambda)=(1-\lambda)^{n_k-n_{k+1}}
r_{\alpha_{k+1}}(\lambda)\le 2^q r_{\alpha_{k+1}}(\lambda).
$$
Thus Assumption \ref{A2.2} is also true.

In order to verify Assumption \ref{A2.3}, we will use some
techniques from \cite{HNS95,K97} and the following well-known
estimates
\begin{equation}\label{5.2}
\|(I-A^*A)^j (A^*A)^\nu\|\le \nu^\nu (j+\nu)^{-\nu}, \quad j\ge 0,
\,\, \nu\ge 0
\end{equation}
for any bounded linear operator $A$ satisfying $\|A\|\le 1$.

For any $\alpha>0$, we set $k:=[1/\alpha]$. Let $A$ and $B$ be any
two bounded linear operators satisfying $\|A\|, \|B\|\le 1$. Then
it follows from (\ref{8.24.2008}) that
\begin{align}\label{5.3}
r_\alpha(A^*A)&-r_\alpha(B^*B) =\sum_{j=0}^k
(I-A^*A)^j\left[A^*(B-A)+(B^*-A^*)B\right](I-B^*B)^{k-j}.
\end{align}
By using (\ref{5.2}) we have
\begin{align*}
\|r_{\alpha}(A^*A)-r_{\alpha}(B^*B)\|
\lesssim & \sum_{j=0}^k \left((j+1)^{-1/2}+(k+1-j)^{-1/2}\right)\|A-B\|\\
\lesssim & \sqrt{k} \|A-B\| \lesssim
\frac{1}{\sqrt{\alpha}}\|A-B\|.
\end{align*}
This verifies (\ref{3.41}).

From (\ref{5.3}) we also have
$A\left[r_{\alpha}(A^*A)-r_\alpha(B^*B)\right]B^*=J_1+J_2$, where
\begin{align*}
J_1&:=\sum_{j=0}^k (I-AA^*)^j AA^*(B-A)(I-B^*B)^{k-j}B^*,\\
J_2&:= \sum_{j=0}^k A(I-A^*A)^j(B^*-A^*)(I-BB^*)^{k-j}BB^*.
\end{align*}
In order to verify (\ref{3.42}), it suffices to show
$\|J_1\|\lesssim (k+1)^{-1/2}\|A-B\|$ since the estimate on $J_2$
is exactly the same. We write $J_1=J_1^{(1)}+J_2^{(2)}$, where
\begin{align*}
J_1^{(1)}&:=\sum_{j=0}^{[k/2]}(I-AA^*)^jAA^* (B-A)(I-B^*B)^{k-j}B^*,\\
J_1^{(2)}&:=\sum_{j=[k/2]+1}^k (I-AA^*)^j AA^*
(B-A)(I-B^*B)^{k-j}B^*.
\end{align*}
With the help of (\ref{5.2}), we can estimate $J_1^{(2)}$ as
\begin{align*}
\|J_1^{(2)}\|&\lesssim \sum_{j=[k/2]+1}^k (j+1)^{-1} (k+j-1)^{-1/2} \|A-B\|\\
&\lesssim (k+1)^{-1} \sum_{j=0}^k (k+1-j)^{-1/2} \|A-B\|\lesssim
(k+1)^{-1/2} \|A-B\|.
\end{align*}
In order to estimate $J_1^{(1)}$, we use $AA^*=I-(I-AA^*)$ to
rewrite it as
\begin{align*}
J_1^{(1)}=&\sum_{j=0}^{[k/2]} (I-AA^*)^j(B-A)(I-B^*B)^{k-j}B^*\\
&-\sum_{j=1}^{[k/2]+1} (I-AA^*)^j(B-A)(I-B^*B)^{k+1-j}B^*\\
=& (B-A)(I-B^*B)^k B^*-(I-AA^*)^{[k/2]+1}(B-A)(I-B^*B)^{k-[k/2]}B^*\\
&+\sum_{j=1}^{[k/2]} (I-AA^*)^j(B-A)(I-B^*B)^{k-j}(B^*B)B^*.
\end{align*}
Thus, in view of (\ref{5.2}), we obtain
\begin{align*}
\|J_1^{(1)}\|\lesssim & (k+1)^{-1/2}\|A-B\|+(k-[k/2]+1)^{-1/2}\|A-B\|\\
&+\sum_{j=1}^{[k/2]} (k-j+1)^{-3/2}\|A-B\|\\
\lesssim &(k+1)^{-1/2}\|A-B\|.
\end{align*}
We thus verify (\ref{3.42}). The verification of (\ref{3.40}) can
be done similarly.

Applying the estimate (\ref{3.40}), we obtain
\begin{align*}
\|\left[g_\alpha(A^*A)-g_\alpha(B^*B)\right]B^*\| &\le
\sum_{j=1}^k\| \left[(I-A^*A)^j-(I-B^*B)^j\right]B^*\|\\
&\lesssim k\|A-B\|\lesssim \frac{1}{\alpha}\|A-B\|,
\end{align*}
which verifies (\ref{3.43}).

Finally we verify Assumption \ref{A6.2} by assuming that $F$
satisfies Assumption \ref{F1} and Assumption \ref{F2}. From
(\ref{5.3}) and Assumption \ref{F1} it follows that
\begin{align*}
r_\alpha(F_x'^* F_x')-r_\alpha(F_z'^*F_z') &= \sum_{j=0}^k
(I-F_x'^*F_x')^j F_x'^*F_x' (R(z, x)-I)
(I-F_z'^*F_z')^{k-j}\\
&+ \sum_{j=0}^k (I-F_x'^*F_x')^j(I-R(x, z))^* F_z'^*F_z'
(I-F_z'^*F_z')^{k-j}.
\end{align*}
Thus we may use the argument in the verification of (\ref{3.42})
to conclude
$$
\|r_\alpha(F_x'^* F_x')-r_\alpha(F_z'^*F_z')\|\lesssim \|I-R(x,
z)\|+\|I-R(z, x)\|\lesssim K_0\|x-z\|.
$$
This verifies (\ref{6.2.1}).

By using (\ref{5.3}) and Assumption \ref{F1} we also have for any
$w\in X$
\begin{align*}
F_x'[r_\alpha (F_x'^*F_x')-r_\alpha(F_z'^*F_z')]w=Q_1+Q_2+Q_3+Q_4,
\end{align*}
where
\begin{align*}
Q_1&=\sum_{j=0}^{[k/2]} (I-F_x'F_x'^*)^j (F_x'F_x'^*) (F_z'-F_x')
(I-F_z'^*F_z')^{k-j}w,\\
Q_2&=\sum_{j=[k/2]+1}^k (I-F_x'F_x'^*)^j (F_x'F_x'^*) (F_z'-F_x')
(I-F_z'^*F_z')^{k-j}w,\\
Q_3&=\sum_{j=0}^{[k/2]} (I-F_x'F_x'^*)^j F_x'(I-R(x, z))^*
(F_z'^*F_z')(I-F_z'^*F_z')^{k-j}w,\\
Q_4&=\sum_{j=[k/2]+1}^k (I-F_x'F_x'^*)^j F_x'(I-R(x, z))^*
(F_z'^*F_z')(I-F_z'^*F_z')^{k-j}w.
\end{align*}
By employing (\ref{5.2}) it is easy to see that
\begin{align*}
\|Q_3\|\lesssim \sum_{j=0}^{[k/2]} (j+1)^{-1/2} (k-j+1)^{-1}
\|I-R(x, z)\|\|w\|\lesssim (k+1)^{-1/2} K_0\|x-z\|\|w\|.
\end{align*}
With the help of (\ref{5.2}) and Assumption \ref{F2}, we have
\begin{align*}
\|Q_2\| & \lesssim \sum_{j=[k/2]+1}^k (j+1)^{-1}
\|(F_z'-F_x')(I-F_z'^*F_z')^{k-j} w\|\\
& \lesssim K_1\|x-z\|\sum_{j=[k/2]+1}^k (j+1)^{-1}(k-j+1)^{-1/2}
\|w\| \\
& +K_2\|F_z'(x-z)\| \sum_{j=[k/2]+1}^k (j+1)^{-1} \|w\|\\
& \lesssim  (k+1)^{-1/2} K_1\|x-z\| \|w\|+ K_2\|F_z'(x-z)\|\|w\|.
\end{align*}
By using the argument in the verification of (\ref{3.42}) and
Assumption \ref{F2} we obtain
\begin{align*}
\|Q_1\| & \lesssim \|(F_z'-F_x')(I-F_z'^*F_z')^k w\|
+\|(F_z'-F_x')(I-F_z'^*F_z')^{k-[k/2]}w\|\\
& + \sum_{j=1}^{[k/2]}
\|(F_z'-F_x')(I-F_z'^*F_z')^{k-j}(F_z'^*F_z') w\|\\
& \lesssim (k+1)^{-1/2} K_1\|x-z\|\|w\|+ K_2\|F_z'(x-z)\|\|w\|\\
& + \sum_{j=1}^{[k/2]} \left(K_1\|x-z\| (k-j+1)^{-3/2}
+K_2\|F_z'(x-z)\| (k-j+1)^{-1}\right)\|w\|\\
& \lesssim (k+1)^{-1/2} K_1\|x-z\|\|w\|+ K_2\|F_z'(x-z)\|\|w\|.
\end{align*}
Using Assumption \ref{F1} and the the similar argument in the
verification of (\ref{3.42}) we also have
\begin{align*}
\|Q_4\|\lesssim (k+1)^{-1/2} \|I-R(x, z)\|\|w\|\lesssim
(k+1)^{-1/2} K_0\|x-z\|\|w\|.
\end{align*}
Combining the above estimates we thus obtain for any $w\in X$
\begin{align*}
\|F_x'[r_\alpha & (F_x'^*F_x')-r_\alpha(F_z'^*F_z')]w\|\\
&\lesssim (K_0+K_1)\alpha^{1/2} \|x-z\|
\|w\|+K_2\|F_z'(x-z)\|\|w\|
\end{align*}
which implies (\ref{6.2.2}).

Therefore, Theorem \ref{T4.1}, Theorem \ref{T4.5} and Theorem
\ref{T8.1} are applicable for the method defined by (\ref{3}) and
(\ref{2.4}) with $g_\alpha$ given by (\ref{g2}).

The similar argument as above also applies to the situation where
$g_\alpha$ is given by
$$
g_\alpha(\lambda):=\sum_{i=0}^{[1/\alpha]}(1+\lambda)^{-i}
$$
which arise from the Lardy's method for solving linear ill-posed
problems.

In summary, we obtain the following result.

\begin{corollary}
Let $F$ satisfy (\ref{2.1}) and (\ref{2.3}), and let
$\{\alpha_k\}$ be a sequence given by $\alpha_k=1/n_k$, where
$\{n_k\}$ is a sequence of positive integers satisfying
(\ref{5.1}) for some $q \ge 1$. Let $\{x_k^\delta\}$ be defined by
(\ref{3}) with
$$
g_\alpha(\lambda)=\sum_{i=0}^{[1/\alpha]}(1-\lambda)^i\qquad
\mbox{or} \qquad
g_\alpha(\lambda)=\sum_{i=0}^{[1/\alpha]}(1+\lambda)^{-i},
$$
and let $k_\delta$ be the first integer satisfying (\ref{2.4})
with $\tau>1$.

(i) If $F$ satisfies Assumption \ref{Lip}, and if $x_0-x^\dag$
satisfies (\ref{7}) for some $\omega\in X$ and $\nu\ge 1/2$, then
$$
\|x_{k_\delta}^\delta-x^\dag\| \le
C_\nu\|\omega\|^{1/(1+2\nu)}\delta^{2\nu/(1+2\nu)}
$$
provided $L\|u\|\le \eta_0$, where $u\in {\mathcal
N}(F'(x^\dag)^*)^\perp\subset Y$ is the unique element such that
$x_0-x^\dag=F'(x^\dag)^* u$, $\eta_0>0$ is a constant depending
only on $\tau$ and $q$, and $C_\nu$ is a constant depending only
on $\tau$, $q$ and $\nu$.

(ii) Let $F$ satisfy Assumption \ref{F1} and Assumption \ref{F2},
and let $x_0-x^\dag \in N(F'(x^\dag))^\perp$. Then there exists a
constant $\eta_1>0$ depending only on $\tau$ and $q$ such that if
$(K_0+K_1+K_2)\|x_0-x^\dag\|\le \eta_1$ then
$$
\lim_{\delta\rightarrow 0} x_{k_\delta}^\delta=x^\dag,
$$
moreover, when $x_0-x^\dag$ satisfies (\ref{7}) for some
$\omega\in X$ and $\nu>0$, then
$$
\|x_{k_\delta}^\delta-x^\dag\|\le C_\nu \|\omega\|^{1/(1+2\nu)}
\delta^{2\nu/(1+2\nu)}
$$
for some constant $C_\nu>0$ depending only on $\tau$, $q$ and
$\nu$; while when $x_0-x^\dag$ satisfies (\ref{logsource}) for
some $\omega\in X$ and $\mu>0$, then
\begin{equation*}
\|x_{k_\delta}^\delta-x^\dag\|\le C_\mu
\|\omega\|\left(1+\left|\ln
\frac{\delta}{\|\omega\|}\right|\right)^{-\mu}
\end{equation*}
for some constant $C_\mu$ depending only on $\tau$, $q$ and $\mu$.
\end{corollary}

\subsection{\bf Example 3} As the last example we consider the method
(\ref{3}) with $g_\alpha$ given by
\begin{equation}\label{g3}
g_\alpha(\lambda)=\frac{1}{\lambda}\left(1-e^{-\lambda/\alpha}\right)
\end{equation}
which arises from the asymptotic regularization for linear
ill-posed problems. In this method, the iterated sequence
$\{x_k^\delta\}$ is equivalently defined as $x_{k+1}^\delta
:=x^\delta(1/\alpha_k)$, where $x^\delta(t)$ is the solution of
the initial value problem
\begin{align*}
&\frac{d}{dt} x^\delta(t)
=F'(x_k^\delta)^*\left(y^\delta-F(x_k^\delta)+
F'(x_k^\delta)(x_k^\delta-x^\delta(t))\right), \quad t>0,\\
&x^\delta(0)=x_0.
\end{align*}
Note that the corresponding residual function is
$$
r_\alpha(\lambda)=e^{-\lambda/\alpha}.
$$
It is easy to see that Assumption \ref{A2.1}(a), (b) and
(\ref{A2.1.2}) hold with
$$
c_0=e^{-1}, \quad c_1=1, \quad c_3=\frac{1}{\sqrt{2e}} \quad
\mbox{and} \quad  c_4=\sqrt{\frac{2}{e}}.
$$
By using the inequality $1-e^{-t}\le \sqrt{t}$ for $t\ge 0$ we
have for $0<\alpha\le \beta$ that
$$
r_\beta(\lambda)-r_\alpha(\lambda) =r_\beta(\lambda)
\left(1-e^{\lambda/\beta-\lambda/\alpha}\right) \le
\sqrt{\frac{\lambda}{\alpha}-\frac{\lambda}{\beta}}
r_\beta(\lambda) \le
\sqrt{\frac{\lambda}{\alpha}}r_\beta(\lambda).
$$
This verifies Assumption \ref{A2.1}(c) with $c_2=1$. It is
well-known that the qualification of the linear asymptotic
regularization is $\bar{\nu}=\infty$ and (\ref{A2.1.3}) is
satisfied with $d_\nu=(\nu/e)^\nu$ for each $0\le \nu< \infty$.

In order to verify Assumption \ref{A2.2}, we assume that
$\{\alpha_k\}$ is a sequence of positive numbers satisfying
\begin{equation}\label{5.21}
0\le \frac{1}{\alpha_{k+1}}-\frac{1}{\alpha_k}\le \theta_0 \quad
\mbox{and} \quad \lim_{k\rightarrow \infty} \alpha_k=0
\end{equation}
for some $\theta_0>0$. Then for all $\lambda\in [0, 1]$ we have
$$
r_{\alpha_k}(\lambda)=e^{(1/\alpha_{k+1}-1/\alpha_k)\lambda}
r_{\alpha_{k+1}}(\lambda) \le e^{\theta_0}
r_{\alpha_{k+1}}(\lambda).
$$
Thus Assumption \ref{A2.2} is also true.

In order to verify Assumption \ref{A2.3} and Assumption
\ref{A6.2}, we set for every integer $n\ge 1$
$$
r_{\alpha, n}(\lambda)
:=\left(1+\frac{\lambda}{n\alpha}\right)^{-n}, \qquad g_{\alpha,
n}(\lambda):=\frac{1}{\lambda}
\left(1-\left(1+\frac{\lambda}{n\alpha}\right)^{-n}\right).
$$
Note that, for each fixed $\alpha>0$, $\{r_{\alpha, n}\}$ and
$\{g_{\alpha, n}\}$ are uniformly bounded over $[0, 1]$, and
$r_{\alpha, n}(\lambda)\rightarrow r_{\alpha}(\lambda)$ and
$g_{\alpha, n}(\lambda)\rightarrow g_{\alpha}(\lambda)$ as
$n\rightarrow \infty$. By the dominated convergence theorem, we
have for any bounded linear operator $A$ with $\|A\|\le 1$ that
\begin{align*}
\lim_{n\rightarrow \infty} &\|[r_{\alpha}(A^*A)-r_{\alpha,
n}(A^*A)]x\|^2\\
&=\lim_{n\rightarrow \infty} \int_0^{\|A\|^2}
\left(r_\alpha(\lambda)-r_{\alpha, n}(\lambda)\right)^2
d(E_\lambda x, x)=0
\end{align*}
and
\begin{align*}
\lim_{n\rightarrow\infty}& \|[g_{\alpha}(A^*A)-g_{\alpha,
n}(A^*A)]x\|^2\\
&=\lim_{n\rightarrow \infty} \int_0^{\|A\|^2}
\left(g_\alpha(\lambda)-g_{\alpha, n}(\lambda)\right)^2
d(E_\lambda x, x)=0
\end{align*}
for any $x\in X$, where $\{E_\lambda\}$ denotes the spectral
family generated by $A^*A$. Thus it suffices to verify Assumption
\ref{A2.3} and Assumption \ref{A6.2} with $g_\alpha$ and
$r_\alpha$ replaced by $g_{\alpha, n}$ and $r_{\alpha, n}$ with
uniform constants $c_6$, $c_7$ and $c_8$ independent of $n$. Let
$A$ and $B$ be any two bounded linear operators satisfying $\|A\|,
\|B\|\le 1$. We need the following inequality which says for any
integer $n\ge 1$ there holds
\begin{equation}\label{5.22}
\|r_{\alpha, n}(A^*A)(A^*A)^\nu\|\le \nu^\nu \alpha^\nu, \qquad
0\le \nu\le n.
\end{equation}
By noting that
\begin{align}\label{6.21.1}
r_{\alpha, n}&(A^*A)-r_{\alpha, n}(B^*B)\nonumber\\
&=\frac{1}{n\alpha}\sum_{i=1}^n r_{\alpha,
i}(A^*A)\left[A^*(B-A)+(B^*-A^*)B\right] r_{\alpha, n+1-i}(B^*B),
\end{align}
we thus obtain
$$
\|r_{\alpha, n}(A^*A)-r_{\alpha, n}(B^*B)\|\le
\sqrt{\frac{2}{\alpha}}\|A-B\|,
$$
\begin{equation}\label{5.23}
\|[r_{\alpha, n}(A^*A)-r_{\alpha, n}(B^*B)]B^*\|\le
\frac{3}{2}\|A-B\|
\end{equation}
and
$$
\|A[r_{\alpha, n}(A^*A)-r_{\alpha, n}(B^*B)]B^*\|\le
\sqrt{2\alpha} \|A-B\|.
$$
Furthermore, by noting that $g_{\alpha,
n}(\lambda)=\frac{1}{n\alpha}\sum_{i=1}^n r_{\alpha, i}(\lambda)$,
we may use (\ref{5.23}) to conclude
\begin{align*}
\|[g_{\alpha, n}(A^*A)-g_{\alpha, n}(B^*B)]B^*\|
& \le  \frac{1}{n\alpha} \sum_{i=1}^n \|[r_{\alpha, i}(A^*A)-r_{\alpha, i}(B^*B)]B^*\|\\
&\le \frac{3}{2\alpha}\|A-B\|.
\end{align*}
Assumption \ref{A2.3} is therefore verified.

It remains to verify Assumption \ref{A6.2} with $g_\alpha$ and
$r_\alpha$ replaced by $g_{\alpha, n}$ and $r_{\alpha, n}$ with
uniform constants $c_7$ and $c_8$ independent of $n$. By using
(\ref{6.21.1}), Assumption \ref{F1} and (\ref{5.22}) we have
\begin{align*}
\|r_{\alpha, n}&(F_x'^*F_x')-r_{\alpha, n}(F_z'^*F_z')\|\\
&\le \frac{1}{n\alpha}\sum_{i=1}^n \|r_{\alpha, i}(F_x'^*F_x')
(F_x'^*F_x') (R(z, x)-I) r_{\alpha,
n+1-i}(F_z'^*F_z')\|\\
&+\frac{1}{n\alpha} \sum_{i=1}^n \|r_{\alpha,
i}(F_x'^*F_x')(I-R(x, z))^*(F_z'^*F_z') r_{\alpha,
n+1-i}(F_z'^*F_z')\|\\
&\le \|I-R(z, x)\|+\|I-R(x, z)\|\\
&\le 2K_0\|x-z\|.
\end{align*}
This implies (\ref{6.2.1}).

By using (\ref{6.21.1}), Assumption \ref{F2} and (\ref{5.22}) we
also have for any $a\in X$ and $b\in Y$ satisfying $\|a\|=\|b\|=1$
that
\begin{align*}
(F_x'[r_{\alpha, n}&(F_x'^*F_x')-r_{\alpha, n}(F_z'^*F_z')]a, b)\\
&\le \frac{1}{n\alpha} \sum_{i=1}^n |(r_{\alpha,
i}(F_x'F_x'^*)(F_x'F_x'^*) (F_z'-F_x') r_{\alpha,
n+1-i}(F_z'^*F_z') a, b)|\\
&+ \frac{1}{n\alpha} \sum_{i=1}^n |(a, r_{\alpha,
n+1-i}(F_z'^*F_z') F_z'^*(F_z'-F_x')F_x'^* r_{\alpha,
i}(F_x'F_x'^*) b)|\\
&\le \sqrt{2} K_1 \alpha^{1/2}
\|x-z\|+K_2\|F_z'(x-z)\|+\frac{1}{2} K_2\|F_x'(x-z)\|.
\end{align*}
This implies (\ref{6.2.2}).

Therefore, we may apply Theorem \ref{T4.1}, Theorem \ref{T4.5} and
Theorem \ref{T8.1} to conclude the following result.

\begin{corollary}
Let $F$ satisfy (\ref{2.1}) and (\ref{2.3}), and let
$\{\alpha_k\}$ be a sequence of positive numbers satisfying
(\ref{5.21}) for some $\theta_0>0$. Let $\{x_k^\delta\}$ be
defined by (\ref{3}) with $g_\alpha$ given by (\ref{g3}) and let
$k_\delta$ be the first integer satisfying (\ref{2.4}) with
$\tau>1$.

(i) If $F$ satisfies Assumption \ref{Lip}, and if $x_0-x^\dag$
satisfies (\ref{7}) for some $\omega\in X$ and $\nu\ge 1/2$, then
$$
\|x_{k_\delta}^\delta-x^\dag\| \le
C_\nu\|\omega\|^{1/(1+2\nu)}\delta^{2\nu/(1+2\nu)}
$$
provided $L\|u\|\le \eta_0$, where $u\in {\mathcal
N}(F'(x^\dag)^*)^\perp\subset Y$ is the unique element such that
$x_0-x^\dag=F'(x^\dag)^* u$, $\eta_0>0$ is a constant depending
only on $\tau$, $\theta_0$ and $\alpha_0$,  and $C_\nu$ is a
constant depending only on $\tau$, $\theta_0$, $\alpha_0$ and
$\nu$.

(ii) Let $F$ satisfy Assumption \ref{F1} and Assumption \ref{F2},
and let $x_0-x^\dag \in N(F'(x^\dag))^\perp$. Then there exists a
constant $\eta_1>0$ depending only on $\tau$, $\theta_0$ and
$\alpha_0$ such that if $(K_0+K_1+K_2)\|x_0-x^\dag\|\le \eta_1$
then
$$
\lim_{\delta\rightarrow 0} x_{k_\delta}^\delta=x^\dag;
$$
moreover, when $x_0-x^\dag$ satisfies (\ref{7}) for some
$\omega\in X$ and $\nu>0$, then
$$
\|x_{k_\delta}^\delta-x^\dag\|\le C_\nu \|\omega\|^{1/(1+2\nu)}
\delta^{2\nu/(1+2\nu)}
$$
for some constant $C_\nu>0$ depending only on $\tau$, $\theta_0$,
$\alpha_0$ and $\nu$; while when $x_0-x^\dag$ satisfies
(\ref{logsource}) for some $\omega\in X$ and $\mu>0$, then
\begin{equation*}
\|x_{k_\delta}^\delta-x^\dag\|\le C_\mu
\|\omega\|\left(1+\left|\ln
\frac{\delta}{\|\omega\|}\right|\right)^{-\mu}
\end{equation*}
for some constant $C_\mu$ depending only on $\tau$, $\theta_0$,
$\alpha_0$ and $\mu$.
\end{corollary}

\begin{acknowledgements}
The authors wish to thank the referee for careful reading of the
manuscript and useful comments.
\end{acknowledgements}


\begin{thebibliography}{999}
%
%
\bibitem{B} A. B. Bakushinskii,  {\it The problems of
the convergence of the iteratively regularized Gauss-Newton
method}, Comput. Math. Math. Phys., 32(1992), 1353--1359.

\bibitem{BH05} F. Bauer and T. Hohage, {\it A Lepskij-type stopping rule for
regularized Newton methods}, Inverse Problems, 21(2005),
1975--1991.

\bibitem{BNS97} B.  Blaschke, A. Neubauer and O. Scherzer,
{\it  On convergence rates for the iteratively regularized
Gauss-Newton method}, IMA J. Numer. Anal., 17(1997), 421--436.

\bibitem{DES98} P. Deuflhard, H. W. Engl and O. Scherzer, {\it A
convergence analysis of iterative methods for the solution of
nonlinear ill-posed problems under affinely invariant conditions},
Inverse Problems, 14 (1998), no. 5, 1081--1106.

\bibitem{H97a} M. Hanke, {\it A regularizing Levenberg-Marquardt scheme
with applications to inverse groundwater filtration problems},
Inverse Problems, 13(1997), 79--95.

\bibitem{H97b} M. Hanke, {\it Regularizing properties of a
truncated Newton-CG algorithm for nonlinear inverse problems},
Numer. Funct. Anal. Optim., 18(1997), 971--993.

\bibitem{HNS95} M. Hanke, A. Neubauer and O. Scherzer, {\it A
convergence analysis of Landweber iteration of nonlinear ill-posed
problems},  Numer. Math., 72(1995),  21--37.

\bibitem{H97} T. Hohage, {\it Logarithmic convergence rates
of the iteratively regularized Gauss-Newton method for an inverse
potential and an inverse scattering problem}, Inverse Problems,  13
(1997),  no. 5, 1279--1299.

\bibitem{Jin00} Q. N. Jin, {\it  On the iteratively
regularized Gauss-Newton method for solving nonlinear ill-posed
problems},  Math. Comp.,  69 (2000),  no. 232, 1603--1623.

\bibitem{Jin08} Q. N. Jin, {\it A convergence analysis of the
iteratively regularized Gauss-Newton method under the Lipschitz
condition}, Inverse Problems, 24(2008), no. 4, to appear.

\bibitem{JH99} Q. N. Jin and Z. Y. Hou, {\it  On an a
posteriori parameter choice strategy for Tikhonov regularization
of nonlinear ill-posed problems},   Numer. Math.,  83(1999),  no.
1, 139--159.

\bibitem{K97} B. Kaltenbacher, {\it Some Newton-type methods
for the regularization of nonlinear illposed problems},  Inverse
Problems, 13(1997), 729--753.

\bibitem{K98} B. Kaltenbacher, {\it A posteriori choice
strategies for some Newton type methods for the regularization of
nonlinear ill-posed problems}, Numer. Math., 79 (1998), 501-528.

\bibitem{KNS08} B. Kaltenbacher, A. Neubauer and O. Scherzer,
{\it Iterative Regularization Methods for Nonlinear Ill-Posed
Problems}, Berlin, de Gruyter, 2008.

\bibitem{R99} A. Rieder, {\it On the regularization of nonlinear
ill-posed problems via inexact Newton iterations},  Inverse
Problems,  15(1999), 309--327.

\bibitem{R01} A. Rieder, {\it On convergence rates of inexact Newton
regularizations},  Numer. Math.,  88(2001), 347--365.

\bibitem{SEK93} O. Scherzer, H. W. Engl and K. Kunisch,
{\it  Optimal a posteriori parameter choice for Tikhonov
regularization for solving nonlinear ill-posed problems}, SIAM J.
Numer. Anal., 30(1993), 1796--1838.

\bibitem{T97} U. Tautenhahn, {\it On a general regularization scheme
for nonlinear ill-posed problems}, Inverse Problems, 13 (1997),
no. 5, 1427--1437.

\bibitem{TJ03} U. Tautenhahn and Q. N. Jin, {\it  Tikhonov
regularization and a posteriori rules for solving nonlinear ill
posed problems},   Inverse Problems,  19 (2003),  no. 1, 1--21.

\bibitem{VV} G. M. Vainikko and A. Y. Veretennikov, {\it Iteration Procedures in
Ill-Posed Problems}, Moscow, Nauka, 1986 (In Russian).
\end{thebibliography}



\end{document}